\documentclass[a4paper,12pt]{article}
\usepackage[text={6.5in,9.5in},centering]{geometry}
\usepackage{graphicx,url,subfig}
\RequirePackage[OT1]{fontenc}
\RequirePackage{amsthm,amsmath,amssymb,amscd}
\RequirePackage[numbers]{natbib}
\RequirePackage[colorlinks,citecolor=blue,urlcolor=blue]{hyperref}
\usepackage{enumerate}
\usepackage{datetime}
\usepackage{etex}
\usepackage[normalem]{ulem} 
\usepackage{soul} 
\makeatother
\numberwithin{equation}{section}
\allowdisplaybreaks

\theoremstyle{plain}
\newtheorem{theorem}{Theorem}[section]
\newtheorem{corollary}[theorem]{Corollary}
\newtheorem{lemma}[theorem]{Lemma}

\theoremstyle{definition}
\newtheorem{definition}[theorem]{Definition}

\newtheorem{remark}[theorem]{Remark}

\newcommand{\E}{\mathbb{E}}

\newcommand{\ud}{\ensuremath{\mathrm{d}}}

\newcommand{\Norm}[1]{\left|\left|  #1   \right|\right|}

\newcommand{\Itos}{It\^{o}'s }

\newcommand{\InPrd}[1]{\left\langle #1 \right\rangle}

\newcommand{\calB}{\mathcal{B}}

\newcommand{\calF}{\mathcal{F}}

\newcommand{\calK}{\mathcal{K}}

\newcommand{\calI}{\mathcal{I}}
\newcommand{\calL}{\mathcal{L}}

\newcommand{\calN}{\mathcal{N}}

\newcommand{\bbN}{\mathbb{N}}

\newcommand*{\one}{{{\rm 1\mkern-1.5mu}\!{\rm I}}}

\newcommand{\bbP}{\mathbb{P}}

\newcommand{\R}{\mathbb{R}}
\newcommand{\RR}{\mathbb{R}}

\def \eref#1{\hbox{(\ref{#1})}}

\DeclareMathOperator{\Lip}{Lip}
\DeclareMathOperator{\LIP}{Lip}

\DeclareMathOperator{\Vip}{\overline{\varsigma}}

\title{
Comparison principle for stochastic heat equation on $\R^d$
}
\author{
{\bf Le Chen}
\quad and\quad   {\bf Jingyu Huang}
\date{\vspace{0em}\small \today}
}

\begin{document}
\maketitle
\begin{abstract}
We establish the strong comparison principle and strict positivity of
solutions to the following nonlinear stochastic heat equation on $\mathbb{R}^d$
\[
\left(\frac{\partial }{\partial t} -\frac{1}{2}\Delta \right) u(t,x) =  \rho(u(t,x))
\:\dot{M}(t,x),
\]
for measure-valued initial data, where $\dot{M}$ is a spatially homogeneous Gaussian noise that is white in time and $\rho$ is Lipschitz continuous. These results are obtained under the condition that 
$\int_{\mathbb{R}^d}(1+|\xi|^2)^{\alpha-1}\hat{f}(\text{d} \xi)<\infty$ for some $\alpha\in(0,1]$,
where $\hat{f}$ is the spectral measure of the noise.
{The weak comparison principle and nonnegativity of solutions to the same equation are obtained 
under Dalang's condition, i.e., $\alpha=0$.}
As some intermediate results, we obtain handy upper bounds for $L^p(\Omega)$-moments of $u(t,x)$ for all $p\ge 2$, and also prove that $u$ is a.s. H\"older continuous with order $\alpha-\epsilon$ in space  and $\alpha/2-\epsilon$ in time for any small $\epsilon>0$. 

\\

\end{abstract}

\setlength{\parindent}{1.5em}



\section{Introduction}

In this paper, we study the {\em sample-path comparison principle}, or simply {\em comparison principle} of the solutions 
to the following stochastic heat equation (SHE) {with {\em rough initial conditions}},
\begin{align}\label{E:SHE}
\begin{cases}
\displaystyle \left(\frac{\partial }{\partial t} -
\frac{1}{2}\Delta \right) u(t,x) =  \rho(u(t,x))
\:\dot{M}(t,x),&
x\in \R^d,\; t>0, \\
\displaystyle \quad u(0,\cdot) = \mu(\cdot).
\end{cases}
\end{align}
In this equation, $\rho$ is assumed to be a globally Lipschitz continuous function.
{The linear case, i.e., $\rho(u)=\lambda u$, is called 
the {\em parabolic Anderson model} (PAM) \cite{CarmonaMolchanov94}}.
The noise $\dot{M}$ is a Gaussian noise that is white in time and homogeneously colored in space.
Informally,
\[
\E\left[\dot{M}(t,x)\dot{M}(s,y)\right] = \delta_0(t-s)f(x-y)
\]
where $\delta_0$ is the Dirac delta measure with unit mass at zero and $f$ is a ``correlation function''
i.e., a nonnegative and nonnegative definite function that is not identically zero.
The Fourier transform of $f$ is denoted by $\hat{f}$
\[
\hat{f}(\xi)= \calF f (\xi) =\int_{\R^d}\exp\left(- i \: \xi\cdot x \right)f(x)\ud x.
\]
In general, $\hat{f}$ is again a nonnegative and nonnegative definite measure, which is usually called  the {\it spectral measure}.
{The precise meaning of the ``rough initial conditions/data'' are specified as follows.} {We first note that by the Jordan decomposition, 
any signed Borel measure $\mu$ can be decomposed as} $\mu=\mu_+-\mu_-$ where
$\mu_\pm$ are two non-negative Borel measures with disjoint support. Denote $|\mu|:= \mu_++\mu_-$.
The rough initial data refers to any signed Borel measure $\mu$ such that 
\begin{align}\label{E:J0finite}
\int_{\R^d} e^{-a |x|^2} |\mu|(\ud x)
<+\infty\;, \quad \text{for all $a>0$}\;,
\end{align}
where $|x|=\sqrt{x_1^2+\dots+x_d^2}$ denotes the Euclidean norm.
It is easy to see that condition \eqref{E:J0finite} is equivalent to the condition
that the solution to the homogeneous equation -- $J_0(t,x)$ defined in \eqref{E:J0} below -- exists for all $t>0$ and $x\in\R^d$.

The comparison principle refers to the property that if two initial conditions are comparable, then 
the corresponding solutions to the stochastic partial differential equations are also comparable. 
For any Borel measure $\mu$ on $\R^d$, ``$\mu\ge 0$'' has its obvious meaning that $\mu$ is a nonnegative measure and ``$\mu> 0$'' refers to the fact that $\mu\ge 0$ and $\mu$ is nonvanishing, i.e., $\mu\ne 0$.
Let $u_1$ and $u_2$ be two solutions starting from two measures $\mu_1$ and $\mu_2$, respectively. 
We say that \eqref{E:SHE} satisfies the {\em weak comparison principle} if $u_1(t,x)\le u_2(t,x)$ a.s. for all $t>0$ and $x\in\R^d$ whenever $\mu_1\le\mu_2$.
Similarly, we say that \eqref{E:SHE} satisfies the {\em strong comparison principle} if $u_1(t,x)< u_2(t,x)$ for all $t>0$ and $x\in\R^d$ a.s. whenever $\mu_1<\mu_2$.
Note that when $\rho(u)=\lambda u$, it is relatively easier to establish the weak comparison principle since the solution can be approximated by its regularized version, which admits a Feynman-Kac formula; see \cite{HHN,HHNT,HLN}.

Most strong comparison principles are obtained through Mueller's original work \cite{Mueller91}, where he proved the case when $d=1$, $\dot{M}$ is the space-time white noise, $\rho(u)=|u|^\gamma$ (for all $\gamma\le 1$), and the initial data is a bounded function. In \cite{Shiga}, Shiga  studied the same equation as that in \cite{Mueller91} except that $\rho$ is assumed to be Lipschitz and there can be a drift term. By using concentration of measure arguments for discrete directed polymers in Gaussian environments, Flores established in \cite{Flores} the strict positivity of solution to 1-d PAM with Dirac delta initial data. 
Following arguments by Mueller and Shiga, Chen and Kim extended these results in \cite{CK14Comp} to allow both fractional Laplace operators and rough initial data. Recently, by using paracontrolled distributions, Gubinelli and Perkowski gave an intrinsic proof of the strict positivity; see \cite{GP15KPZ}.
Their proof does not depend on the details of noise, though they require the initial data to be a function that is strict positive anywhere.

When $d\ge 2$, in order to study a random field solution, the noise has to have some color in space.
Equation \eqref{E:SHE} has been much studied since the introduction by Dawson and Salehi \cite{DS80} {as a model for the growth of a population in a random environment.} 
In \cite{Dalang99, DalangQuer}, it is shown that if the initial condition is a bounded function, and  under some integrability condition on $\hat{f}$, now called {\it Dalang's} condition, i.e., 
\begin{align}\label{E:Dalang}
\Upsilon(\beta):=(2\pi)^{-d}\int_{\R^d} \frac{\hat{f}(\ud \xi)}{\beta+|\xi|^2}<+\infty \quad \text{for some and hence for all $\beta>0$,}
\end{align}
there is a unique random field solution to equation \eqref{E:SHE}. 
This equation has been extensively studied; see, e.g., \cite{CHKK16Bad,FK13SHE,HHN,HLN}.
Recently, Chen and Kim showed that Dalang's condition \eqref{E:Dalang} also guarantees
an $L^2(\Omega)$-continuous random field solution starting from rough initial conditions; see \cite{CK15SHE}. 
To the best of our knowledge, comparison principle in this setting is much less known, though people believe that it is true.
In \cite{TessitoreZabczyk98}, 
Tessitore and Zabczyk proved the strict positivity for the case when $\hat{f}$ belongs to $L^p(\R^d)$ for some $p\in [1,d/(d-2))$. Clearly, this condition excludes the important Riesz kernel case, i.e., $f(x)=|x|^{-\beta}$ with $\beta\in (0,2\wedge d)$.
Indeed, we will show that under Dalang's condition \eqref{E:Dalang}, if $\rho(0)=0$,
then the solution $u(t,x)$ starting from any nonnegative rough initial data is a.s. nonnegative for any $t>0$ and $x\in\R^d$. 
Moreover, if the nonnegative rough initial data is nonvanishing and $f$ satisfies 
\begin{align}\label{E:DalangAlpha}
\int_{\R^d}\frac{\hat{f}(\ud\xi)}{\left(1+|\xi|^2\right)^{1-\alpha}}<\infty, \quad\text{for some $\alpha\in(0,1]$,}
\end{align}
then we are able to establish the strict positivity of $u(t,x)$ through the following small-ball probability estimate:
\[
\bbP \left( u(t,x)< \epsilon \right) \leq A \exp \left( -A |\log \epsilon|^{\alpha} \left(\log |\log \epsilon|\right)^{1+\alpha} \right)\,.
\]
Similar small-ball probabilities in various settings can be found in \cite{CK14Comp,CJK12,Flores,MuellerNualart}.
These nonnegativity statements can be translated into comparison statements by 
considering $v=u_1-u_2$.

Condition \eqref{E:DalangAlpha} is natural since in a recent paper \cite{CHKK16Bad}, it is shown that Dalang's condition \eqref{E:Dalang} alone cannot guarantee the existence of a continuous version of the solution.
There might be solutions that behave so badly that they may hit zero. Whether this phenomenon does happen is still not clear to us and it is left for future exploration. For the moment, we are content with this slightly strong condition \eqref{E:DalangAlpha}. 
Indeed, if the initial condition is a bounded function, Sanz-Sol\'e and Sarr\`a \cite{SanzSoleSarra02} showed that
condition \eqref{E:DalangAlpha} guarantees that 
the solution is a.s. H\"older continuous
with order $\alpha-\epsilon$ in space  and $\alpha/2-\epsilon$ in time for any small $\epsilon>0$.
In this paper, we have extended this result for rough initial conditions.
The space-time white noise case is proved in \cite{ChenDalangHolder}.

In all these studies, the moment bounds/formulas play an important role. The upper bounds for the second moments under Dalang's condition \eqref{E:Dalang} for rough initial conditions is obtained in \cite{CK15SHE}. In this paper, we extend this bound to obtain similar upper bounds for all $p$-th moments, $p\ge 2$. Using these moments upper bounds, we establish the (weak) comparison principle, whose property is assumed in \cite{CK15SHE} in order to obtain some nontrivial lower bounds for the second moments. 
Note when $\rho(u)=\lambda u$, the $p$-th moment admits a Feynman-Kac representation, which has been exploited to study the intermittency phenomenon in \cite{HHN,HHNT,HLN}.

%


\subsection{Main results}
The solution to \eqref{E:SHE} is understood as the mild form
\begin{align}\label{E:mild}
u(t,x)&= J_0(t,x) + \int_0^t \int_{\R^d} G(t-s,x-y)\rho(u(s,y))M(\ud s,\ud y),
\end{align}
where $J_0(t,x)$ is the solution to the homogeneous equation
\begin{align}\label{E:J0}
J_0(t,x):= \left(\mu* G(t,\cdot)\right)(x)= \int_{\R^d}G(t,x-y)\mu(\ud y)
\end{align}
and
\begin{align}
G(t,x)= \left(2\pi t\right)^{-d/2} \exp\left(-\frac{|x|^2}{2t}\right).
\end{align}

We will prove seven theorems listed as follows:

\begin{theorem}[Weak comparison principle]
\label{T:WComp}
Assume that $f$ satisfies Dalang's condition \eqref{E:Dalang}.
Let $u_1(t,x)$ and $u_2(t,x)$ be two solutions to \eqref{E:SHE}
with the initial measures $\mu_1$ and $\mu_2$ that satisfy \eqref{E:J0finite}, respectively. 
If $\mu_1 \le \mu_2$, then
\begin{equation}\label{E: W comp point}
\bbP \left(u_1(t,x)\leq u_2(t,x)\right)=1\,, \quad \text{for all}\  t \geq 0 \ \text{and}\  x\in \RR^d\,.
\end{equation} 
Moreover, if the paths of $u_1(t,x)$ and $u_2(t,x)$ are a.s. continuous, then 
\begin{equation}\label{E: W comp path}
\bbP \left(u_1(t,x)\leq u_2(t,x) \ \text{for all}\  t \geq 0 \ \text{and}\  x\in \RR^d\right)=1\,. 
\end{equation} 
\end{theorem}     

{
If $\rho(0)=0$, then $u\equiv 0$ is the unique solution to \eqref{E:SHE} starting from $\mu=0$. Hence, we have the following corollary:
\begin{corollary}[Nonnegativity]
Assume that $f$ satisfies Dalang's condition \eqref{E:Dalang} and $\rho(0)=0$.
Let $u$ be the solution to \eqref{E:SHE}
with the initial measure $\mu$ that satisfies \eqref{E:J0finite}. 
If $\mu \ge 0$, then
\begin{equation}\label{E:Nonneg}
\bbP \left(u(t,x)\ge 0\right)=1\,, \quad \text{for all}\  t \geq 0 \ \text{and}\  x\in \RR^d\,.
\end{equation} 
Moreover, if the path of $u(t,x)$ are a.s. continuous, then 
\begin{equation}\label{E:Nonneg_path}
\bbP \left(u(t,x)\geq 0 \ \text{for all}\  t \geq 0 \ \text{and}\  x\in \RR^d\right)=1\,. 
\end{equation}  
\end{corollary}
}

\begin{theorem}[Strong comparison principle]
\label{T:SComp}
Assume that $f$ satisfies \eqref{E:DalangAlpha} for some $\alpha\in(0,1]$.
Let $u_1(t,x)$ and $u_2(t,x)$ be two solutions to \eref{E:SHE} with the initial data $\mu_1$ and $\mu_2$, respectively. 
Then the fact that $\mu_1 < \mu_2$ implies 
\begin{equation}
\bbP \left(u_1(t,x)< u_2(t,x)\ \text{for all}\ t>0 \ \text{and}\  x \in \RR^d \: \right)=1.
\end{equation}
\end{theorem}

\begin{theorem}[Strict positivity]
\label{T:Pos}
Assume that $f$ satisfies \eqref{E:DalangAlpha} for some $\alpha\in(0,1]$ and $\rho(0)=0$.
Let $u$ be the solution to \eqref{E:SHE} with initial measure $\mu>0$
that satisfies \eqref{E:J0finite}.
Then for any compact set $K \subset \RR_+^*\times \RR^d$, 
there exists a finite constant $A>0$ which only depends on $K$ such that for all $\epsilon >0$ small enough,
\begin{equation}\label{E:SmallBall}
\bbP \left( \inf_{(t,x)\in K} u(t,x)< \epsilon \right) \leq A \exp \left( -A |\log \epsilon|^{\alpha} \left(\log |\log \epsilon|\right)^{1+\alpha} \right)\,.
\end{equation}
\end{theorem}


\bigskip
In order to establish the above results, we need to prove the following four theorems, which are of interest by themselves. 
The first result is a general moment bound. This provides us with a very handy tool in studying various properties of the solution to \eqref{E:SHE}. This result extends the previous work \cite{CK15SHE} from the two-point correlation function to higher moments.
{Let $\LIP_\rho>0$ be the Lipschitz constant for $\rho$.}
See Section \ref{S:Mom} for the proof.

\begin{theorem}[Moment bounds]
\label{T:Mom}
Under Dalang's condition \eqref{E:Dalang},
if the initial data $\mu$ is a signed measure that satisfies \eqref{E:J0finite},
then the solution $u$ to \eqref{E:SHE} for any given $t>0$ and $x\in\R^d$ is in $L^p(\Omega)$, $p\ge 2$, and
{
\begin{align}\label{E:Mom}
 \Norm{u(t,x)}_p \le
 \sqrt{2}\left[\Vip+\sqrt{2} \left(|\mu|*G(t,\cdot)\right)(x) \right]H\left(t;\gamma_p\right)^{1/2},
\end{align}
where $\Vip=|\rho(0)|/\Lip_\rho$ and} $\gamma_p=32p\Lip_\rho^2$ and $H(t;\gamma_p)$ is defined in \eqref{E:H} below.
Moreover, if for some $\alpha\in (0,1]$ condition \eqref{E:DalangAlpha} is satisfied,
then when $p\ge 2$ is large enough, there exists some constant $C>0$ such that 
{
\begin{align}\label{E:MomAlpha}
\Norm{u(t,x)}_p \leq C \Big[\: \Vip+\left(|\mu|*G(t,\cdot)\right)(x) \Big] \exp\left(C p^{1/\alpha} t\right)\,.
\end{align}}
\end{theorem}

The second result is about the sample-path regularity under condition \eqref{E:DalangAlpha} for rough initial conditions.
This result is used to obtain a large deviation estimates in proving the strong comparison principle.
See Section \ref{S:Holder} for its proof.

\begin{theorem}[H\"older regularity]
\label{T:Holder}
Suppose that $\mu$ is any measure that satisfies \eqref{E:J0finite}
and $f$ satisfies \eqref{E:DalangAlpha} for some $\alpha\in (0,1]$.
Then the solution to \eqref{E:SHE} starting from $\mu$ is a.s.
$\beta_1$-H\"older continuous in time and
$\beta_2$-H\"older continuous in space on $(0,\infty)\times\R^d$ with
\[
\beta_1\in \left(0,\alpha/2\right)\quad
\text{and}\quad
\beta_2\in \left(0,\alpha\right).
\]
\end{theorem}

\bigskip
The third theorem consists of two approximation results, which are used to establish the weak comparison principle.
The first one says that we can approximate a solution starting from a rough initial data by solutions starting from smooth and bounded initial conditions.
This result allows us to pass from the weak comparison principle for $L^\infty(\R^d)$-valued initial data to that for rough initial data.
In the second approximation, we mollify the noise and establish an uniform $L^2(\Omega)$-limit. See Section \ref{S:Approx} for the proof.

\begin{theorem}[Two approximations]
\label{T:Approx}
{Assume that $f$ satisfies Dalang's condition \eqref{E:Dalang}.}

{\noindent (1)} Suppose that the initial measure $\mu$ satisfies \eqref{E:J0finite}.
If $u(t,x)$ and $u_{\epsilon}(t,x)$ be the solutions to \eqref{E:SHE} starting from
$\mu$ and
$((\mu\:\psi_\epsilon)*G(\epsilon,\cdot))(x)$, respectively, where
\begin{align}\label{E:psi}
 \psi_\epsilon(x) = \one_{\{|x|\le 1/\epsilon\}} + \left(1+1/\epsilon - |x| \right)\one_{\{1/\epsilon<|x|\le 1+1/\epsilon\}},
\end{align}
then
{
\[
\lim_{\epsilon\rightarrow 0_+}\Norm{u(t,x)-u_{\epsilon}(t,x)}_2 =0,\quad \text{for all $t>0$ and $x\in\R^d$.}
\]}
{\noindent (2)} 
Let $\phi$ be any {continuous, nonnegative and nonnegative definite function} on $\R^d$ 
with compact support such that $\int_{\R^d}\phi(x)\ud x=1$. 
Let $u$ be the solution to \eqref{E:SHE} starting from a bounded initial data,
i.e., $\mu(\ud x)=g(x)\ud x$ with $g\in L^\infty(\R^d)$.
If $u^\epsilon$ is the solution to the following mollified equation
\begin{equation}\label{E:SHE regularize}
\frac{\partial}{\partial t} u_\epsilon(t,x) =\frac{1}{2}\Delta u_\epsilon (t,x)+\rho(u_\epsilon (t,x))\dot{M}^{\epsilon}(t,x)\,,
\end{equation}
with the same initial condition $u^\epsilon(0,\cdot)=\mu$ as $u$, 
where 
\begin{equation}
M^{\epsilon}(\ud s, \ud x) = \int_{\RR^d} \phi_{\epsilon}(x-y) M(\ud s,\ud y)\ud x \,, 
\end{equation}
and $\phi_\epsilon(x)=\epsilon^{-d}\phi(x/\epsilon)$,
then
\begin{align}
 \lim_{\epsilon\rightarrow 0_+}\sup_{x\in\R^d}\Norm{u(t,x)-u_{\epsilon}(t,x)}_2 =0,\quad 
 \text{for all $t>0$.}
\end{align}
\end{theorem}

\begin{remark}\label{R:Example}
One can always find one example of such function $\phi$ in part (2) of Theorem \ref{T:Approx},
e.g.,
$\phi(x)=\prod_{i=1}^d \left(1-|x_i|\right)\one_{\{|x_i|\le 1\}}$ 
whose Fourier transform is nonnegative: $\hat{\phi}(\xi)=2^{d}\prod_{j=1}^d \xi_j^{-2}(1-\cos(\xi_j))\ge 0$.
\end{remark}

\bigskip
The last result shows that the solution $u(t,x)$ to \eqref{E:SHE} converges to its initial data $\mu$ weakly as
$t\rightarrow 0$.
This result is used to establish the strong comparison principle for measure-valued initial data
given that for function-valued initial data.
See Section \ref{S:WeakSol} for the proof.
Let $C_c(\R^d)$ be the set of continuous functions with compact support.

\begin{theorem}\label{T:WeakSol}
{Under Dalang's condition \eqref{E:Dalang}, 
if $u(t,x)$  be the solution to \eqref{E:SHE} starting from a Borel measure $\mu$ that satisfies \eqref{E:J0finite},
then }
\begin{align}
\lim_{t\rightarrow 0} \int_{\R^d} u(t,x) \phi(x)\ud x = \int_{\R^d} \phi(x)\mu(\ud x)
\qquad\text{in $L^2(\Omega)$ for all $\phi\in C_c(\R^d)$,}
\end{align}
where the probability space is introduced in Section \ref{S:Notation}.
\end{theorem}

\bigskip
This paper is organized as follows:
After some preliminaries in Section \ref{S:Notation},
we first prove the moment bounds, Theorem \ref{T:Mom}, in Section \ref{S:Mom}.
Using these moment bounds, we proceed to establish the H\"older regularity,
Theorem \ref{T:Holder}, in Section \ref{S:Holder}.
Then in Section \ref{S:Approx} we prove Theorem \ref{T:Approx} for the two approximations.
The weak limit as $t$ goes to zero, i.e., Theorem \ref{T:WeakSol}, is proved in Section \ref{S:WeakSol}. With these preparation, we prove the weak comparison principle, Theorem \ref{T:WComp}, in Section \ref{S:WComp}. Finally, in Section \ref{S:SComp} we prove both the strong comparison principle (Theorem \ref{T:SComp}) and the strict positivity (Theorem \ref{T:Pos}). Some technical lemmas are given in Appendix. Throughout this paper, $C$ will denote a generic constant which may vary {at each occurrence}.

\section{Some preliminaries}
\label{S:Notation}

\subsection{Definition and existence of a solution}
Recall that a {\em spatially homogeneous Gaussian noise that is white in time} is an
$L^2(\Omega)$-valued mean zero Gaussian process on a complete probability space $\left(\Omega,\calF,\bbP\right)$
\[
\left\{F(\psi):\: \psi\in C_c^{\infty}\left([0,\infty)\times\R^{d}\right)\:\right\},
\]
such that
\[
\E\left[F(\psi)F(\phi)\right] = \int_0^{\infty} \ud s\iint_{\R^{2d}}\psi(s,x)\phi(s,y)f(x-y)\ud x\ud y.
\]
Let $\calB_b(\R^d)$ be the collection of Borel measurable sets with finite Lebesgue measure.
As in Dalang-Walsh theory \cite{Dalang99,Walsh},
one can extend $F$ to a $\sigma$-finite $L^2(\Omega)$-valued martingale measure $B\mapsto F(B)$
defined for $B\in \calB_b(\R_+\times\R^d)$. Then define
\[
M_t(B) :=F\left([0,t]\times B \right), \quad B\in\calB_b(\R^d).
\]
Let $(\calF_t,t\ge 0)$ be the natural filtration generated by $M_\cdot(\cdot)$
and augmented by all $\bbP$-null sets $\calN$ in $\calF$, i.e.,
\[
\calF_t := \sigma\left(M_s(A):\: 0\le s\le t,
A\in\calB_b\left(\R^d\right)\right)\vee
\calN,\quad t\ge 0,
\]
Then for any adapted, jointly measurable (with respect to
$\calB\left((0,\infty)\times\R^d\right)\times\calF$) random field $\{X(t,x): t>0,x\in\R^d\}$ such that
\[
\int_0^\infty\ud s\iint_{\R^{2d}}\ud x\ud y\:
\Norm{X(s,y)X(s,x)}_{\frac{p}{2}} f(x-y) <\infty,
\]
the stochastic integral
\[
\int_0^\infty \int_{\R^d} X(s,y)M(\ud s,\ud y)
\]
is well-defined in the sense of Dalang-Walsh. Here we only require the joint-measurability instead of
predictability; see Proposition 2.2 in \cite{CK15SHE} for this case or
Proposition 3.1 in \cite{ChenDalang13Heat} for the space-time white noise case.
Throughout this paper, $\Norm{\cdot}_p$ denotes the $L^p(\Omega)$-norm.

\bigskip


We formally write the SPDE \eqref{E:SHE} in the integral form
\begin{align}\label{E:WalshSI}
u(t,x)= J_0(t,x)+ I(t,x)
\end{align}
where
\[
I(t,x):= \iint_{[0,t]\times \R^d} G(t-s,x-y) \rho(u(s,y))M(\ud s,\ud y).
\]
The above stochastic integral is understood in the sense of Walsh \cite{Dalang99,Walsh}.

\begin{definition}\label{D:Solution}
A process $u=\left(u(t,x),\:(t,x)\in(0,\infty)\times\R^d \right)$  is called a {\it
random field solution} to \eqref{E:SHE} if
\begin{enumerate}[(1)]
 \item $u$ is adapted, i.e., for all $(t,x)\in(0,\infty)\times\R^d$, $u(t,x)$ is
$\calF_t$-measurable;
\item $u$ is jointly measurable with respect to
$\calB\left((0,\infty)\times\R^d\right)\times\calF$;
\item $\Norm{I(t,x)}_2<+\infty$ for all $(t,x)\in(0,\infty)\times\R^d$;
\item  $I$ is $L^2(\Omega)$-continuous, i.e., the function $(t,x)\mapsto I(t,x)$ mapping $(0,\infty)\times\R^d$ into
$L^2(\Omega)$ is continuous;
\item $u$ satisfies \eqref{E:WalshSI} a.s.,
for all $(t,x)\in(0,\infty)\times\R^d$.
\end{enumerate}
\end{definition}

\begin{theorem}[Theorem 2.4 in \cite{CK15SHE}]
\label{T:ExUni}
If the initial data $\mu$ satisfies \eqref{E:J0finite},
then under Dalang's condition \eqref{E:Dalang},
SPDE \eqref{E:SHE} has a unique (in the sense of versions) random field solution
$\left\{u(t,x): t>0,x \in\R^d\right\}$ starting from $\mu$.
This solution is $L^2(\Omega)$-continuous.
\end{theorem}

\subsection{Some special functions}
We first introduce some notation following \cite{CK15SHE}.
Denote
\begin{align}\label{E:k}
k(t):=\int_{\R^d}f(z)G(t,z)\ud z.
\end{align}
By Fourier transform, this function can be written in the following form
\begin{align}\label{E:k2}
k(t):=(2\pi)^{-d} \int_{\R^d}\hat{f}(\ud\xi)\exp\left(-\frac{t|\xi|^2}{2}\right).
\end{align}
Define $h_0(t):=1$ and for $n\ge 1$,
\begin{align}\label{E:hn}
 h_n(t)= \int_0^t \ud s \: h_{n-1}(s) k(t-s).
\end{align}
Let
\begin{align}\label{E:H}
 H(t;\gamma):= \sum_{n=0}^\infty \gamma^n h_n(t).
\end{align}
This function is defined through the correlation function $f$. The following lemma
tells us that this function has an exponential bound.

\begin{lemma}[Lemma 2.5 in \cite{CK15SHE} or Lemma 3.8 in \cite{BC16}]
\label{L:EstHt}
For all $t\ge 0$ and $\gamma\ge 0$,
\begin{align}
\label{E:Var} 
\limsup_{t\rightarrow\infty} \frac{1}{t}\log H(t;\gamma)\le \inf\left\{\beta>0:  \:\Upsilon\left(\beta\right) < \frac{1}{\gamma}\right\}.
\end{align}
\end{lemma}

\subsection{A remark on two recursions}

The purpose of this part is to compare the two recursions \eqref{E:IntIneq1}
and \eqref{E:TwoPoint1} below.
While Lemma \ref{L:IntIneq} gives an easy to use upper bound,
Lemma \ref{L:TwoPoint} is sharper and used in \cite{CK15SHE} to obtain lower bounds for the second moment. {Recursion \eqref{E:TwoPoint1} and its conclusion \eqref{E:TwoPoint2} will play a crucial role in the proof of part (1) of Theorem \ref{T:Approx}.}

In order to make this statement clear, we need to introduce some notation.
For $h,w:\R_+\times\R^{3d}\mapsto\R$,
define the (asymmetric convolution) operation ``$\rhd$'', which depends on $f$, as follows
\begin{equation}
\label{E:RHD-BACK}
\begin{aligned}
\left(h \rhd w\right)(t,x,x';y):=\int_0^t \ud s \iint_{\R^{2d}}\ud z \ud z'\; h(t-s,x-z,x'-z';y-(z-z'))&  \\
\qquad \times  w(s,z,z';y)\: f(y-(z-z'))& .
\end{aligned}
\end{equation}
By change of variables,
\begin{equation}
\label{E:RHD-FOR}
\begin{aligned}
\left(h \rhd w\right)(t,x,x';y):=\int_0^t \ud s \iint_{\R^{2d}}\ud z \ud z'\: h(s,z,z';y-[(x-z)-(x'-z')])&\\
\times  w(t-s,x-z,x'-z';y)\: f(y-[(x-z)-(x'-z')])&.
\end{aligned}
\end{equation}
This operation is associative (see Lemma B.1 in \cite{CK15SHE})
\[
\left(\left(h\rhd w\right)\rhd v\right)(t,x,x';y)
=\left(h\rhd \left(w\rhd v\right)\right)(t,x,x';y).
\]
We use the convention that if a function $h$ is defined on $\R_+\times\R^{2d}$ instead of
$\R_+\times\R^{3d}$, when applying the operation $\rhd$ to $h$,
it is meant for $h'(t,x,x';y):=h(t,x,x')$.

\bigskip
For $t>0$ and  $x,x',y\in\R^d$, define recursively:
\[
\calL_{0}(t,x,x';y):= G(t,x)G(t,x')
\]
and for $n\ge 1$,
\[
\calL_{n}(t,x,x';y):= \left(\calL_{0}\rhd \calL_{n-1}\right)(t,x,x';y).
\]
For $\lambda\in\R$,
Lemma 2.7 of \cite{CK15SHE} ensures that the following series is well defined
\begin{align}\label{E:calK}
\calK_{\lambda}(t,x,x';y):= \sum_{n=0}^\infty \lambda^{2(n+1)}\calL_{n}(t,x,x';y)
\le \calL_0(t,x,x') H(t;2\lambda^2).
\end{align}

Then the upper bounds for the two-point correlation function in Theorem 2.4 of \cite{CK15SHE}
can be summarized as the following lemma.

\begin{lemma}\label{L:TwoPoint}
If for some nonnegative function $J_*:\R_+\times\R^{2d}\mapsto\R_+$ and $\lambda\ge 0$, a function $g:\R_+\times\R^{2d}\mapsto\R$ satisfies the following integral inequality
\begin{align}\label{E:TwoPoint1}
\begin{aligned}
 g(t,x,x')\le J_*(t,x,x')+ \lambda^2 \int_0^t\ud s\iint_{\R^{2d}} &
 G(t-s,x-y_1) G(t-s,x'-y_2) \\
 &\times f(y_1-y_2) g(s,y_1,y_2)\ud y_1\ud y_2,
\end{aligned}
\end{align}
then
\begin{align}\label{E:TwoPoint3}
 g(t,x,x')\le J_*(t,x,x')+\left(\calK_\lambda\rhd J_*\right)(t,x,x';0).
\end{align}
In particular,
\begin{align}\label{E:TwoPoint2}
\begin{aligned}
 g(t,x,x')\le J_*(t,x,x')+ H(t;2\lambda^2)\int_0^t\ud s\iint_{\R^{2d}} &
 G(t-s,x-y_1) G(t-s,x'-y_2) \\
 &\times f(y_1-y_2) J_*(s,y_1,y_2)\ud y_1\ud y_2.
\end{aligned}
\end{align}
When the inequality in \eqref{E:TwoPoint1} is equality, the conclusion in \eqref{E:TwoPoint3} is also an equality.
\end{lemma}
\begin{proof}
This lemma is proved using the Picard iteration.
We need only to prove the case when both inequalities in \eqref{E:TwoPoint1} and \eqref{E:TwoPoint2}
are equalities.
Notice that \eqref{E:TwoPoint1} (with inequality replaced by equality) can be written as
\[
g(t,x,x')=J_*(t,x,x')+ \lambda^2 (\calL_0\rhd g)(t,x,x';0).
\]
Let
\[
g_0(t,x,x')
:=J_*(t,x,x'),
\]
and for $n\ge 1$,
\begin{align}
 g_n(t,x,x')
 = J_*(t,x,x') + \lambda^2 \left(\calL_0\rhd g_{n-1}\right)(t,x,x';0).
\end{align}
Then by the associativity of the operator $\rhd$, we see that
\begin{align*}
  g_n(t,x,x')
 = J_*(t,x,x') + \sum_{k=0}^n \lambda^{2(k+1)} \left(\calL_k\rhd J_*\right)(t,x,x';0).
\end{align*}
Therefore,
\begin{align*}
  g(t,x,x')=\lim_{n\rightarrow\infty}
  g_n(t,x,x')
 &= J_*(t,x,x') + \sum_{k=0}^\infty \lambda^{2(k+1)} \left(\calL_k\rhd J_*\right)(t,x,x';0)\\
 &= J_*(t,x,x') + \left(\calK_{\lambda} \rhd J_*\right)(t,x,x';0)\\
 &\le
J_*(t,x,x') + H(t;2\lambda^2)\left(\calL_0 \rhd J_*\right)(t,x,x';0),
\end{align*}
where in the last step we have applied the bound for $\calK$ in \eqref{E:calK}.
This proves Lemma \ref{L:TwoPoint}.
\end{proof}

\section{Moment bounds (Proof of Theorem \ref{T:Mom})}
\label{S:Mom}

While Lemma \ref{L:TwoPoint} is appropriate for dealing with the two-point correlation function,
the corresponding recursion for the $p$-point $(p>2)$ correlation function will be much more complicated.
We will instead consider the bounds for the $p$-th moment.
The following lemma will play the same role to the $p$-th moment as Lemma \ref{L:TwoPoint} to the two-point correlation function.

\begin{lemma}\label{L:IntIneq}
Suppose that $\mu$ is a signed measure that satisfies \eqref{E:J0finite} and
let $J_0(t,x)$ be the solution to the homogeneous equation (see \eqref{E:J0}).
If a nonnegative function $g:\R_+\times\R^{d}\mapsto\R_+$ satisfies the following integral inequality
\begin{align}\label{E:IntIneq1}
\begin{aligned}
 g(t,x)^2\le J_0^2(t,x)+ \lambda^2 \int_0^t\ud s\iint_{\R^{2d}} &
 G(t-s,x-y_1) G(t-s,x-y_2) \\
 &\times f(y_1-y_2) g(s,y_1)g(s,y_2)\ud y_1\ud y_2,
\end{aligned}
\end{align}
for all $t>0$ and $x\in\R^d$, then{
\begin{align}\label{E:IntIneq2}
\begin{aligned}
 g(t,x)\le \sqrt{2} \: \left(|\mu|*G(t,\cdot)\right)(x) \: H(t;2\lambda^2)^{1/2}.
\end{aligned}
\end{align}}
\end{lemma}
\begin{proof}
We prove this lemma using the Picard iteration.
As the proof of Lemma \ref{L:TwoPoint}, we need only
to prove the case when the inequality in \eqref{L:IntIneq} is an equality.
Let{
\[
g_0(t,x) = \left(|\mu|*G(t,\cdot)\right)(x),
\]}
and for $n\ge 1$,
\begin{align}\label{E:gn}
\begin{aligned}
 g_{n}^2(t,x) = J_0^2(t,x) + \lambda^2 \int_0^t\ud s\iint_{\R^{2d}} & G(t-s,x-y_1)G(t-s,x-y_2) f(y_1-y_2)\\
 &\times g_{n-1}(s,y_1)g_{n-1}(s,y_2)\ud y_1 \ud y_2.
\end{aligned}
\end{align}
For $\gamma=2\lambda^2$, we claim that{
\begin{align}\label{E2:Indt-Lp}
g_n(t,x)\le \sqrt{2}\: g_0(t,x) \left(\sum_{i=0}^n \gamma^i h_i(t)\right)^{1/2},\quad\text{for all $n\ge 0$.}
\end{align}}
It is clear that \eqref{E2:Indt-Lp} holds  for $n=0$.
Suppose that \eqref{E2:Indt-Lp} is true for $n\ge 0$.
Notice that
\begin{align*}
g_{n+1}^2(t,x)
=& J_0^2(t,x)
+\lambda^2 \int_0^t \iint_{\R^{2d}}
G(t-s,x-y_1)
G(t-s,x-y_2) f(y_1-y_2)\\
&\hspace{10em}\times g_n(s,y_1) g_n(s,y_2) \ud s\ud y_1\ud y_2\\
=: & J_0^2(t,x) + \lambda^2\:  I(t,x).
\end{align*}
By the induction assumption,
\begin{align*}
I(t,x)\le&
2\int_0^t \ud s\iint_{\R^{2d}} \ud y_1 \ud y_2\:f(y_1-y_2)
G(t-s,x-y_1)
G(t-s,x-y_2) \\
&\times
|J_0(s,y_1)| \:
|J_0(s,y_2)|
\left(\sum_{i=0}^n \gamma^i h_i(s)\right) \\
=&
2\int_0^t \ud s\iint_{\R^{2d}}\ud y_1\ud y_2\:f(y_1-y_2)
G(t-s,x-y_1)
G(t-s,x-y_2) \\
&
{
\times 
\iint_{\R^{2d}}|\mu|(\ud z_1)|\mu|(\ud z_2)G(s,y_1-z_1)G(s,y_2-z_2) \left(\sum_{i=0}^n \gamma^i h_i(s)\right).}
\end{align*}
Because (see \cite[Lemma 5.4]{ChenDalang13Heat})
\begin{align}\label{E:GGGG}
G(s,x)G(t-s,y) =
G\left(\frac{s(t-s)}{t},\frac{sy-(t-s)x}{t}\right)G(t,x+y),
\end{align}
we see that
\[
G(t-s,x-y_1)
G(s,y_1-z_1) = G(t,x-z_1) G\left(\frac{s(t-s)}{t}, y_1-z_1-\frac{s}{t}(x-z_1)\right).
\]
Hence,
\begin{align*}
I(t,x) \le &
2\int_0^t \ud s \left(\sum_{i=0}^n \gamma^i h_i(s)\right) \iint_{\R^{2d}}\ud y_1\ud y_2\:f(y_1-y_2)\\
&\times G\left(\frac{s(t-s)}{t}, y_1-z_1-\frac{s}{t}(x-z_1)\right)\\
&\times
G\left(\frac{s(t-s)}{t}, y_2-z_2-\frac{s}{t}(x-z_2)\right)\\
&
{
\times \iint_{\R^{2d}}|\mu|(\ud z_1)|\mu|(\ud z_2) \:G(t,x-z_1)G(t,x-z_2).}
\end{align*}
By Fourier transform, we see that
the above double integral $\ud y_1\ud y_2$ is equal to
\begin{align*}
(2\pi)^{-d}
\int_{\R^d}
\hat{f}(\ud \xi)&
\exp\left(i \frac{t-s}{t}(z_1-z_2) \cdot \xi - \frac{s(t-s)}{t}|\xi|^2\right).
\end{align*}
Since $f$ is nonnegative and nonnegative definite, this integral is bounded by
\begin{align*}
(2\pi)^{-d}\int_{\R^d}
\hat{f}(\ud \xi)
\exp\left(- \frac{s(t-s)}{t}|\xi|^2\right).
\end{align*}
Hence,
\begin{align}
 \notag
 I(t,x)\le&
 2\ g_0^2(t,x)
 \int_0^t\ud s
 \left(\sum_{i=0}^n \gamma^i h_i(s)\right)
 (2\pi)^{-d}\int_{\R^d}
\hat{f}(\ud \xi)
\exp\left(- \frac{s(t-s)}{t}|\xi|^2\right).
\end{align}
Then using the fact that $t\rightarrow h_i(t)$ is nondecreasing (see Lemma 2.6 in \cite{CK15SHE}),
by Lemma \ref{L:shatf} with $\beta=|\xi|^2/2$, we see that
\begin{align*}
I(t,x)
\le &4\ g_0^2(t,x)
\int_{0}^t\ud s
 \left(\sum_{i=0}^n \gamma^i h_i(s)\right)
(2\pi)^{-d}\int_{\R^d}
\hat{f}(\ud \xi)
\exp\left(- \frac{t-s}{2}|\xi|^2\right).
\end{align*}
Then by \eqref{E:k2} and \eqref{E:hn}, we see that
\begin{align*}
I(t,x) 
& \le 4\ g_0^2(t,x)
\int_0^t
\ud s
 \left(\sum_{i=0}^n \gamma^i h_i(s)\right)
k(t-s)\\
&= 4\ g_0^2(t,x) \sum_{i=0}^n \gamma^i h_{i+1}(t).
\end{align*}
Therefore,
\begin{align*}
g_{n+1}^2(t,x) &\le
g_0^2(t,x) + 4\lambda^2 g_0^2(t,x) \sum_{i=0}^n \gamma^i h_{i+1}(t)\le 2J_0^2(t,x) \sum_{i=0}^{n+1} \gamma^i h_{i}(t).
\end{align*}
This proves \eqref{E2:Indt-Lp}.
Finally,
\begin{align*}
g(t,x)\le \lim_{n\rightarrow\infty}  \sqrt{2} g_0(t,x) \left(\sum_{i=0}^{n} \gamma^i h_{i}(t)\right)^{1/2}
=\sqrt{2} \ g_0(t,x) \left(\sum_{i=0}^{\infty} \gamma^i h_{i}(t)\right)^{1/2},
\end{align*}
which completes the proof of Lemma \ref{L:IntIneq}.
\end{proof}

\bigskip

\begin{proof}[Proof of Theorem \ref{T:Mom}]
The unique solution in $L^2(\Omega)$ has been established in \cite{CK15SHE}.
We will prove the moment bounds  in three steps.

{\bigskip\noindent\bf Step 1.~}
Now we prove this moment bound using the Picard iteration.
Let
\[
u_0(t,x) = J_0(t,x),
\]
and for $n\ge 1$,
\begin{align}\label{E:Un}
u_{n}(t,x) = J_0(t,x) + \int_0^t\int_{\R^d} G(t-s,x-y) \rho(u_{n-1}(s,y))M(\ud s,\ud y).
\end{align}
{Since $\rho$ is Lipschitz, by denoting $\Vip=|\rho(0)|/\Lip_\rho$,} 
\[
\Norm{\rho\left(X\right)}_p
\le \Lip_\rho \Norm{\Vip + |X|\: }_p
\le \Lip_\rho \sqrt{2 \left(\Vip^2 + \Norm{X}_p^2 \right)}\:.
\]
Because by the Burkholder-Davis-Gundy inequality and
linear growth condition of $\rho$,
\begin{align*}
\Vip^2+\Norm{u_{n+1}(t,x)}_p^2
\le& \Vip^2+ 2J_0^2(t,x)
+8p \int_0^t \iint_{\R^{2d}}
G(t-s,x-y_1)
G(t-s,x-y_2) f(y_1-y_2)\\
&\hspace{6em}\times \Norm{\rho(u_n(s,y_1))}_{p}\Norm{\rho(u_n(s,y_2))}_{p} \ud s\ud y_1\ud y_2\\
\le& \Vip^2+2J_0^2(t,x)
+16p \Lip_\rho^2 \int_0^t \iint_{\R^{2d}}
G(t-s,x-y_1)
G(t-s,x-y_2) f(y_1-y_2)\\
&\hspace{6em} \times
\sqrt{\Vip^2+\Norm{u_n(s,y_1)}_{p}^2}
\sqrt{\Vip^2+\Norm{u_n(s,y_2)}_{p}^2}
\ud s\ud y_1\ud y_2,
\end{align*}
we can apply the same induction arguments as those in the proof of Lemma \ref{L:IntIneq}
with $\lambda^2=16p\Lip_\rho^2$ and $g_n(t,x)=\sqrt{\Vip^2+\Norm{u_n(t,x)}_p^2}$ and $J_0(t,x)$ replaced by $\Vip+\sqrt{2} J_0(t,x)$ to conclude that
{
\begin{align}\notag
\Norm{u_n(t,x)}_p& \le  \sqrt{\Vip^2+\Norm{u_n(t,x)}_p^2}\\
&\le
\sqrt{2} \left(\Vip+\sqrt{2} \: \left(|\mu|*G(t,\cdot)\right)(x)\right) \left(\sum_{i=0}^n \left(32p\Lip_\rho^2\right)^i h_i(t)\right)^{1/2},
\label{E:Indt-Lp}
\end{align}}
for all $n\ge 0$.

{\bigskip\noindent\bf Step 2.~}
In this step, we will show that
$\{u_n(t,x),\:n\in\bbN\}$ defined in \eqref{E:Un} is a Cauchy sequence in $L^p(\Omega)$.
{Without loss of generality, we may assume that $\mu\ge 0$, otherwise one may simply replace $\mu$ by $|\mu|$ at each occurrence of $\mu$.}
This will then imply the moment bound in \eqref{E:Mom}.
Denote
\[
F_n(t,x)=\Norm{u_{n+1}(t,x)-u_{n}(t,x)}_p.
\]
Then
\begin{align*}
F_n^2(t,x)\le& 8p\Lip_\rho^2 \int_0^t\ud s\iint_{\R^{2d}} \ud y_1\ud y_2\:
G(t-s,x-y_1)G(t-s,x-y_2)\\
&\times f(y_1-y_2) F_{n-1}(s,y_1)F_{n-1}(s,y_2),
\end{align*}
for $n\ge 1$, and
\begin{align*}
F_0^2(t,x) =& \Norm{u_1(t,x)-J_0(t,x)}_p^2
\\
\le& 8p\LIP_\rho^2 \int_0^t\ud s\iint_{\R^{2d}} \ud y_1\ud y_2\:
G(t-s,x-y_1)G(t-s,x-y_2)\\
&\times f(y_1-y_2) J_0(s,y_1)J_0(s,y_2).
\end{align*}
Then by setting $F_{-1}(t,x):= J_0(t,x)$ and $\gamma=16p\LIP_\rho^2$, we see that one can apply the
same induction arguments in the proof of Lemma \ref{L:IntIneq} to conclude that
\begin{align*}
 \sum_{n=0}^\infty F_n(t,x)\le \sqrt{2} J_0(t,x) \left(\sum_{i=0}^\infty \gamma^i h_i(t)\right)^{1/2}<\infty.
\end{align*}
Therefore, $\{u_n(t,x),\:n\in\bbN\}$ is a Cauchy sequence in $L^p(\Omega)$ and
\begin{align*}
\Norm{u(t,x)}_p &
= \lim_{n\rightarrow\infty} \Norm{u_n(t,x)}_p\\
&\le \lim_{n\rightarrow\infty}
\sqrt{2}\left(\Vip+\sqrt{2} J_0(t,x) \right) \left(\sum_{i=0}^n (32p\Lip_\rho^2)^i h_i(t)\right)^{1/2}\\
&=
\sqrt{2}\left(\Vip+\sqrt{2} J_0(t,x) \right) H\left(t;32p\Lip_\rho^2\right)^{1/2}<\infty.
\end{align*}
This proves \eqref{E:Mom}.

{\bigskip\noindent\bf Step 3.~}
In this step, we will prove \eqref{E:MomAlpha}.
Notice that in this case for $\beta>0$,
\begin{align*}
\Upsilon(\beta)&=(2\pi)^{-d}\int_{\R^d} \frac{1}{\left(\beta+|\xi|^2\right)^{\alpha}} \frac{\hat{f}(\ud \xi)}{\left(\beta+|\xi|^2\right)^{1-\alpha}} \\
&\le \frac{C}{\beta^\alpha}
\left(\int_{|\xi|\le 1} \frac{\hat{f}(\ud \xi)}{\beta^{1-\alpha}} + \int_{|\xi|> 1} \frac{\hat{f}(\ud \xi)}{|\xi|^{2(1-\alpha)}}\right)
\le  C \left(\frac{1}{\beta}+\frac{1}{\beta^\alpha}\right).
\end{align*}
From now on fix the constant $C$ on the right-hand side of the above inequalities.
If $p$ is large enough such that $32p\Lip_\rho^2 C>1$, then
\[
C \left(\frac{1}{\beta}+\frac{1}{\beta^\alpha}\right)\le \frac{1}{32p\Lip_\rho^2}
\quad \Longleftarrow\quad
\frac{2 C}{\beta^\alpha} \le \frac{1}{32p\Lip_\rho^2}
\quad \Longleftrightarrow\quad
\beta\ge \left(C 64 p \Lip_\rho^2\right)^{1/\alpha} =:\beta_p.
\]
Then an application of Lemma \ref{L:EstHt} shows that
\[
\limsup_{t\rightarrow\infty}\frac{1}{t}\log H(t;32p\Lip_\rho^2) \le \beta_p.
\]
Hence, the function $e^{-\beta_p t}H(t;32p\Lip_\rho^2)$ is a continuous function on $[0,\infty]$.
Therefore, for some constant $C'>0$, $e^{-\beta_p t}H(t;32p\Lip_\rho^2)\le C'$ for all $t\ge 0$.
This proves \eqref{E:MomAlpha} and also completes the whole proof of Theorem \ref{T:Mom}.
\end{proof}

\section{H\"older regularity (Proof of Theorem \ref{T:Holder})}
\label{S:Holder}

We first prove one lemma.

\begin{lemma}\label{L:GG-xt}
For all $\alpha\in(0,1]$, $x,y\in\R^d$ and $t'\ge t>0$, we have that
\begin{align}\label{E:GG-x}
 \left|G(t,x)-G(t,y)\right|\le\frac{C}{t^{\alpha/2}}\left[G(2t,x)+G(2t,y)\right]|x-y|^{\alpha},
\end{align}
and
\begin{align}\label{E:GG-t}
 \left|G(t,x)-G(t',x)\right| \le C t^{-\alpha/2}G\left(4t',x\right)
 (t'-t)^{\alpha/2}.
\end{align}
\end{lemma}
\begin{proof}
We first prove \eqref{E:GG-x}.
By the scaling property, it suffices to prove that
\[
 \left|G(1,x)-G(1,y)\right|\le \left[G(2,x)+G(2,y)\right]|x-y|^{\alpha}.
\]
We may assume that $|x|\leq |y|$. Choosing $\bar{x}\in \RR^d$ such that $|\bar{x}|=|x|$ and $y = a \bar{x}$ for some $a\geq 1$, i,e, $\bar x$, $y$ and the origin are on the same line. By the mean-value theorem, for some $c\in[0,1]$ and $\xi = c\bar{x} +(1-c)y$,
\begin{align*}
 \left|G(1,x)-G(1,y)\right| =& \left|G(1,\bar{x})-G(1,y)\right|  \le G(1,\xi) |\xi| |\bar{x}-y|
 \le C G(2,\xi)  |\bar{x}-y|.
\end{align*}
Then by the choice of $\bar{x}$, we see that
\[
 G(2,\xi)  |\bar{x}-y|
\le C \left[G(2,\bar{x})+G(2,y)\right]|\bar{x}-y| 
 \le C \left[G(2,x)+G(2,y)\right]|x-y|\,.
\]
Therefore,
\begin{align*}
\left|G(1,x)-G(1,y)\right| & =
\left|G(1,x)-G(1,y)\right|^\alpha \left|G(1,x)-G(1,y)\right|^{1-\alpha}\\
&\le C \left[G(2,x)+G(2,y)\right]^\alpha |x-y|^\alpha
\left|G(2,x)+G(2,y)\right|^{1-\alpha}\\
&=
C \left[G(2,x)+G(2,y)\right]|x-y|^\alpha,
\end{align*}
which proves \eqref{E:GG-x}.

As for \eqref{E:GG-t}, notice that
\begin{align*}
 |G(t,x)-G(t',x)|
 \le & (2\pi)^{-d/2} \left|t^{-d/2}-(t')^{-d/2}\right|e^{-\frac{|x|^2}{2t}}
 +(2\pi)^{-d/2} (t')^{-d/2}\left|e^{-\frac{|x|^2}{2t}}-e^{-\frac{|x|^2}{2t'}}\right|\\
 =&
  t^{d/2} \left|t^{-d/2}-(t')^{-d/2}\right|G(t,x)
 + (t')^{-d/2}\left|G\left(1,\frac{x}{\sqrt{t}}\right)-G\left(1,\frac{x}{\sqrt{t'}}\right)\right|.
\end{align*}
For any $\gamma\in(0,1)$, because $t'>t$,
\begin{align}
\notag
 \left|t^{-d/2}-(t')^{-d/2}\right| &=
 \left|t^{-d/2}-(t')^{-d/2}\right|^{1-\gamma}
 \left|t^{-d/2}-(t')^{-d/2}\right|^{\gamma}\\
\notag
 &\le C \left[2 t^{-d/2}\right]^{1-\gamma}\left[\left(t^{-d/2-1}+(t')^{-d/2-1}\right)|t-t'|\right]^\gamma\\
 &\le C t^{-d/2-\gamma} |t-t'|^\gamma.
 \label{E:t t' increment}
\end{align}
By \eqref{E:GG-x}, for all $\alpha\in(0,1]$,
\begin{align*}
 \left|G\left(1,\frac{x}{\sqrt{t}}\right)-G\left(1,\frac{x}{\sqrt{t'}}\right)\right|
 \le& C \left[G\left(2,\frac{x}{\sqrt{t}}\right)+G\left(2,\frac{x}{\sqrt{t'}}\right)\right]
 |x|^\alpha \left|t^{-1/2}-(t')^{-1/2}\right|^\alpha\\
 \le& C G\left(2,\frac{x}{\sqrt{t'}}\right)
 |x|^\alpha \left|t^{-1/2}-(t')^{-1/2}\right|^\alpha\\
 = &
 C (t')^{d/2}G\left(2t',x\right)
 |x|^\alpha \left|t^{-1/2}-(t')^{-1/2}\right|^\alpha.
\end{align*}
By the subadditivity of $\sqrt{x}$, we see that
\[
|t^{-1/2}-(t')^{-1/2}| = \frac{\sqrt{t'} - \sqrt{t}}{\sqrt{t t'}}
=
\frac{\sqrt{t'-t+t} - \sqrt{t}}{\sqrt{t t'}}
\le
\frac{\sqrt{t'-t}+\sqrt{t} -\sqrt{t}}{\sqrt{t t'}} =
\frac{\sqrt{t'-t}}{\sqrt{t t'}}.
\]
Hence,
\begin{align*}
 \left|G\left(1,\frac{x}{\sqrt{t}}\right)-G\left(1,\frac{x}{\sqrt{t'}}\right)\right|
 \le &
 C t^{-\alpha/2}(t')^{(d-\alpha)/2}G\left(2t',x\right)
 |x|^\alpha (t'-t)^{\alpha/2}\\
 \le &
 C t^{-\alpha/2}(t')^{d/2}G\left(4t',x\right)
 (t'-t)^{\alpha/2}.
\end{align*}
The bound in \eqref{E:GG-t} is proved by taking $\gamma=\alpha/2$ in \eqref{E:t t' increment} and using the fact that
$G(t,x)\le C G(4t',x)$.
This completes the proof of Lemma \ref{L:GG-xt}.
\end{proof}

\bigskip
\begin{proof}[Proof of Theorem \ref{T:Holder}.]
Denote the stochastic integral in \eqref{E:mild} by $\calI(t,x)$.
{Set $\Vip=|\rho(0)|/\Lip_\rho$.}
We need only to prove the H\"older regularity for $\calI(t,x)$.
Fix $n>1$.
For all $(t,x)$ and $(t',x')\in [1/n,n]\times \R^d$ with $t'>t$,
we see that
\begin{align*}
 \Norm{\calI(t,x)-\calI(t',x')}_p^2
 \le C I_1(t,x,x') + C I_2(t,t',x') + C I_3(t,t',x'),
\end{align*}
where
\begin{align}
\label{E:HolderI1}
\begin{aligned}
I_1(t,x,x') =&\int_0^t\ud s\iint_{\R^{2d}}\ud y_1\ud y_2 \:f(y_1-y_2)\\
&\times \left|G(t-s,x-y_1)-G(t-s,x'-y_1)\right| \sqrt{\Vip^2 + \Norm{u(s,y_1)}_p^2} \\
&\times \left|G(t-s,x-y_2)-G(t-s,x'-y_2)\right| \sqrt{\Vip^2 + \Norm{u(s,y_2)}_p^2}
\end{aligned}
\end{align}
and
\begin{align}
 \label{E:HolderI2}
 \begin{aligned}
I_2(t,t',x') =&\int_0^t\ud s\iint_{\R^{2d}}\ud y_1\ud y_2\:f(y_1-y_2) \\
&\times \left|G(t-s,x'-y_1)-G(t'-s,x'-y_1)\right| \sqrt{\Vip^2 + \Norm{u(s,y_1)}_p^2}\\
&\times \left|G(t-s,x'-y_2)-G(t'-s,x'-y_2)\right| \sqrt{\Vip^2 + \Norm{u(s,y_2)}_p^2}
\end{aligned}
\end{align}
and
\begin{align}
 \label{E:HolderI3}
\begin{aligned}
I_3(t,t',x') =&\int_t^{t'}\ud s\iint_{\R^{2d}}\ud y_1\ud y_2\:f(y_1-y_2)\\
&\times G(t'-s,x'-y_1) \sqrt{\Vip^2 + \Norm{u(s,y_1)}_p^2} \\
&\times G(t'-s,x'-y_2)\sqrt{\Vip^2 + \Norm{u(s,y_2)}_p^2}.
\end{aligned}
\end{align}
Note that when $\Vip\ne 0$, from the moment bounds  in \eqref{E:Mom}, by choosing
\[
\widetilde{\mu}(\ud x) = \sqrt{2}\mu(\ud x) + \Vip\ud x\quad\text{and}\quad\widetilde{J_0}(t,x):=\sqrt{2}J_0(t,x)+\Vip\,,
\]
one can reduce it to the case that $\Vip=0$, i.e., $\rho(0)=0$. Hence, in the following, we only need to consider the case that $\Vip=0$.
We will study  these three increments in three steps.

{\bigskip\noindent\bf Step 1.~}
In this step, we study $I_1$.
We apply the moment bound \eqref{E:Mom} to \eqref{E:HolderI1}, it follows that
\begin{align*}
I_1(t,x,x')\leq C H(t,\gamma_p)&\int_0^t\ud s\iint_{\R^{2d}}\ud y_1\ud y_2\:f(y_1-y_2)\\
&\times \left|G(t-s,x-y_1)-G(t-s,x'-y_1)\right|\\
&\times
\left|G(t-s,x-y_2)-G(t-s,x'-y_2)\right|\\
&\times \iint_{\RR^{2d}} \mu(\ud z_1) \mu(\ud z_2)G(s,y_1-z_1)G(s,y_2-z_2)\,.
\end{align*}
Here we have used the definition of $J_0(t,x)$ and the fact that $H(s,\gamma_p)$ is nondecreasing in $s$, see Lemma 2.6 in \cite{CK15SHE}.
By Lemma \ref{L:GG-xt} and and \eqref{E:GGGG}, for all $\alpha\in (0,1)$,
\begin{align*}
\big|G(t-s,x-y_1)&-G(t-s,x'-y_1)\big|\\
\le&
C\left[G(2(t-s),x-y_1)+G(2(t-s),x'-y_1)\right]
\frac{\left|x-x'\right|^\alpha}{(t-s)^{\alpha/2}},
\end{align*}
and
\begin{align*}
G(s,y_1-z_1)&\left|G(t-s,x-y_1)-G(t-s,x'-y_1)\right|\\
& \le
CG(2s,y_1-z_1)\left[G(2(t-s),x-y_1)+G(2(t-s),x'-y_1)\right]
\frac{\left|x-x'\right|^\alpha}{(t-s)^{\alpha/2}}\\
&=
C \frac{\left|x-x'\right|^\alpha}{(t-s)^{\alpha/2}}
\Bigg[
G(2t,x-z_1)G\left(\frac{2s(t-s)}{t},y_1-z_1-\frac{s}{t}(x-z_1)\right)\\
&\hspace{7em}+
G(2t,x'-z_1)G\left(\frac{2s(t-s)}{t},y_1-z_1-\frac{s}{t}(x'-z_1)\right)
\Bigg].
\end{align*}
A similar bound holds for the expression with respect to $y_2$ and $z_2$. Expanding the product of the two bounds, we will get a sum of four terms, 
\[
I_1(t,x,x')\le \sum_{k=1}^ 4 I_{1,k}(t,x,x'),
\]
where, for example,
\begin{align*}
 I_{1,1}(t,x,x')\le & C |x-x'|^{2\alpha} \iint_{\R^{2d}}\mu(\ud z_1)\mu(\ud z_2) \int_0^t\ud s\iint_{\R^{2d}}
 \ud y_1\ud y_2 \: f(y_1-y_2) \frac{1}{(t-s)^{\alpha}}\\
 &\times G(2t,x-z_1)G\left(\frac{2s(t-s)}{t},y_1-z_1-\frac{s}{t}(x-z_1)\right)\\
 &\times
 G(2t,x'-z_2)G\left(\frac{2s(t-s)}{t},y_2-z_2-\frac{s}{t}(x-z_2)\right),
\end{align*}
and similarly for $I_{1,i}$, $i=2, 3,4$.
Because
\[
\left|\calF [G(t,\cdot+w)](\xi)\right|\le \exp\left(-\frac{t}{2}|\xi|^2\right),\quad \text{for all $w\in\R^d$},
\]
we see that
\begin{align*}
I_{1,1}(t,x,x')\le & C |x-x'|^{2\alpha} \iint_{\R^{2d}}\mu(\ud z_1)\mu(\ud z_2) \int_0^t\ud s\int_{\R^{d}}
 \hat{f}(\ud \xi) \: \frac{1}{(t-s)^{\alpha}}\\
 &\times G(2t,x-z_1)G(2t,x-z_2) \exp\left(-\frac{2s(t-s)}{t}|\xi|^2\right)\\
 =&C |x-x'|^{2\alpha} J_0(2t,x)J_0(2t,x')
 \int_0^t\ud s\int_{\R^{d}}
 \hat{f}(\ud \xi) \: \frac{\exp\left(-\frac{2s(t-s)}{t}|\xi|^2\right)}{(t-s)^{\alpha}}.
\end{align*}
By Lemma \ref{L:shatf} with $g(s)=s^{-1/\alpha}$ and $\beta=|\xi|^2$ ($g$ is nonincreasing),
\begin{align*}
 \int_0^t\ud s\int_{\R^{d}}
 \hat{f}(\ud \xi) \: \frac{\exp\left(-\frac{2s(t-s)}{t}|\xi|^2\right)}{(t-s)^{\alpha}}
 & \le
 2\int_{0}^t\ud s\int_{\R^{d}}
 \hat{f}(\ud \xi) \: \frac{1}{s^{\alpha}}\exp\left(-s|\xi|^2\right)\\
 &\le
 2e^t\int_{0}^t\ud s\int_{\R^{d}}
 \hat{f}(\ud \xi) \: \frac{1}{s^{\alpha}}\exp\left(-s(|\xi|^2+1)\right)\\
 &\le C\int_{\R^{d}}
 \frac{\hat{f}(\ud \xi)}{(1+|\xi|^2)^{1-\alpha}}.
\end{align*}
Hence, 
\[
I_{1,1}(t,x,x')\le
C |x-x'|^{2\alpha} J_0(2t,x)J_0(2t,x')
\int_{\R^{d}}
\frac{\hat{f}(\ud \xi)}{(1+|\xi|^2)^{1-\alpha}}.
\]
One can obtain similar bounds  for all the other three terms. Therefore,
\[
I_{1}(t,x,x')\le
C |x-x'|^{2\alpha} [J_0(2t,x)+J_0(2t,x')]^2
\int_{\R^{d}}
\frac{\hat{f}(\ud \xi)}{(1+|\xi|^2)^{1-\alpha}}.
\]

{\noindent\bigskip\bf Step 2.~}
Now we consider the time increment $I_2$.
By the moment bound \eqref{E:Mom},
\begin{align*}
I_2(t,t',x') \le & C \int_0^t\ud s\iint_{\R^{2d}}\ud y_1\ud y_2\:
\left|G(t-s,x'-y_1)-G(t'-s,x'-y_1)\right|\\
&\times
\left|G(t-s,x'-y_2)-G(t'-s,x'-y_2)\right|
f(y_1-y_2)
J_0(s,y_1)J_0(s,y_2)\\
= & C \iint_{\R^d} \mu(\ud z_1)\mu(\ud z_2)
\int_0^t\ud s\iint_{\R^{2d}}\ud y_1\ud y_2\:
f(y_1-y_2)
\\
&\times
G(s,y_1-z_1)\left|G(t-s,x'-y_1)-G(t'-s,x'-y_1)\right|\\
&\times
G(s,y_2-z_2) \left|G(t-s,x'-y_2)-G(t'-s,x'-y_2)\right|.
\end{align*}
Applying \eqref{E:GG-t}, using the fact that $G(s,y_1-z_1)\le C G(4s,y_1-z_1)$ and then applying \eqref{E:GGGG},
we see that
\begin{align*}
  G(s,y_1-z_1) &|G(t-s,x'-y_1)-G(t'-s,x'-y_1)|\\
  \le& C (t-s)^{-\alpha/2} G(s,y_1-z_1) G(4(t'-s),x'-y_1)  \left(t'-t\right)^{\alpha/2}\\
  \le& C (t-s)^{-\alpha/2} G(4s,y_1-z_1) G(4(t'-s),x'-y_1)  \left(t'-t\right)^{\alpha/2}\\
 \le& C (t-s)^{-\alpha/2} G(4t',x'-z_1)G\left(\frac{4s(t'-s)}{t'},y_1-z_1-\frac{s}{t'}(x'-z_1)\right)  \left(t'-t\right)^{\alpha/2}.
\end{align*}
Therefore,
\begin{align*}
I_2(t,t',x')
\le &C (t'-t)^\alpha \iint_{\R^{2d}}\mu(\ud z_1)\mu(\ud z_2)
G(4t',x'-z_1)G(4t',x'-z_2)\\
&\times
\int_0^t\ud s\: (t-s)^{-\alpha}\iint_{\R^{2d}}\ud y_1\ud y_2\:
f(y_1-y_2)
\\
  &\times G\left(\frac{4s(t'-s)}{t'},y_1-z_1-\frac{s}{t'}(x'-z_1)\right)G\left(\frac{4s(t'-s)}{t'},y_2-z_2-\frac{s}{t'}(x'-z_2)\right)\\
  \le &
  C (t'-t)^\alpha \iint_{\R^{2d}}\mu(\ud z_1)\mu(\ud z_2)
G(4t',x'-z_1)G(4t',x'-z_2)\\
&\times
\int_0^t\ud s\: (t-s)^{-\alpha}\int_{\R^d} \hat{f}(\ud \xi) \exp\left(-\frac{4s(t-s)}{t}|\xi|^2\right)\\
=& C(t'-t)^\alpha J_0^2(4t',x')
\int_{\R^d} \hat{f}(\ud \xi) \int_0^t\ud s\: (t-s)^{-\alpha} \exp\left(-\frac{4s(t-s)}{t}|\xi|^2\right),
\end{align*}
where in the second inequality above we have used the fact that
\begin{equation*}
\exp \left(-\frac{4s (t'-s)}{t'} |\xi|^2 \right) \leq \exp \left( -\frac{4s (t-s)}{t} |\xi|^2\right)
\end{equation*}
since $t'\geq t$.
By the same arguments as those in Step 1,
\begin{align*}
 I_{2}(t,t',x')&\le
 C(t'-t)^\alpha J_0^2(4t',x') \int_{\R^{d}}\hat{f}(\ud \xi) \int_0^t\ud s\:
 s^{-\alpha} \exp\left(-2s|\xi|^2\right)\\
 &\le
 C(t'-t)^\alpha J_0^2(4t',x') \int_{\R^{d}}\hat{f}(\ud \xi) \int_0^t\ud s\:
 s^{-\alpha} \exp\left(-2s\left(1+|\xi|^2\right)\right)\\
 &\le
 C(t'-t)^\alpha J_0^2(4t',x') \int_{\R^{d}}\frac{\hat{f}(\ud \xi)}{(1+|\xi|^2)^{1-\alpha}}\,.
\end{align*}

{\noindent\bigskip\bf Step 3.~}
Now we consider the time increment $I_3$.
By the moment bound \eqref{E:Mom} and \eqref{E:GGGG},
\begin{align*}
I_3(t,t',x') \le & C \int_t^{t'}\ud s\iint_{\R^{2d}}\ud y_1\ud y_2\:
G(t'-s,x'-y_1) G(t'-s,x'-y_2) f(y_1-y_2)J_0(s,y_1)J_0(s,y_2)\\
= &C \iint_{\R^d} \mu(\ud z_1)\mu(\ud z_2)
\int_t^{t'}\ud s\iint_{\R^{2d}}\ud y_1\ud y_2\:
f(y_1-y_2)
\\
&\times
G(s,y_1-z_1)G(t'-s,x'-y_1)G(s,y_2-z_2) G(t'-s,x'-y_2)\\
= &
C \iint_{\R^d} \mu(\ud z_1)\mu(\ud z_2) G(t',x'-z_1)G(t',x'-z_2)
\int_t^{t'}\ud s\iint_{\R^{2d}}\ud y_1\ud y_2\:
f(y_1-y_2)\\
& \times G\left(\frac{s(t'-s)}{t'}, y_1-z_1-\frac{s}{t'}(x'-z_1)\right)
G\left(\frac{s(t'-s)}{t'}, y_2-z_2-\frac{s}{t'}(x'-z_2)\right)\\
\le &
C \iint_{\R^d} \mu(\ud z_1)\mu(\ud z_2) G(t',x'-z_1)G(t',x'-z_2)\int_t^{t'}\ud s\int_{\R^{d}}\hat{f}(\ud \xi)
\\
& \times \exp\left(-\frac{s(t'-s)}{t'}|\xi|^2\right)\\
=&
C J_0^2(t',x')
\int_t^{t'}\ud s\int_{\R^{d}}\hat{f}(\ud \xi)
\exp\left(-\frac{s(t'-s)}{t'}|\xi|^2\right).
\end{align*}
Notice that for any $\alpha\in(0,1]$,
\begin{align*}
 \int_t^{t'}\ud s \exp\left(-\frac{s(t'-s)}{t'}|\xi|^2\right)&\le
 \int_t^{t'}\ud s \exp\left(-\frac{t(t'-s)}{t'}|\xi|^2\right)\\
 &\le \int_t^{t'}\ud s \exp\left(-\frac{t(t'-s)}{t'}\left(1+|\xi|^2\right) + \frac{t(t'-t)}{t'}\right)\\
 &\le C \int_t^{t'}\ud s \exp\left(-\frac{t(t'-s)}{t'}\left(1+|\xi|^2\right) \right)\\
 &=C \frac{1-\exp\left(-\frac{t(t'-t)}{t'}\left(1+|\xi|^2\right)\right)}{1+|\xi|^2}\\
 &\le  C \frac{\left(\frac{t(t'-t)}{t'}\left(1+|\xi|^2\right)\right)^\alpha }{1+|\xi|^2}
 =  \frac{C (t'-t)^\alpha}{\left(1+|\xi|^2\right)^{1-\alpha}}.
\end{align*}
Therefore,
\begin{align}\label{E_:HolderI3}
I_3(t,t',x')\le C (t'-t)^{\alpha} J_0^2(t',x')
\int_{\R^{d}}\frac{\hat{f}(\ud \xi)}{(1+|\xi|^2)^{1-\alpha}}.
\end{align}
Combining these three cases and applying the Kolmogorov's continuity theorem, we have completed the proof of Theorem \ref{T:Holder}.
\end{proof}

\section{One approximation result (Proof of Theorem \ref{T:Approx})}
\label{S:Approx}

\begin{proof}[Proof of Theorem \ref{T:Approx}]
{\bf\noindent (1)~}
By Theorem \ref{T:ExUni}, we see that both $u$ and $u_\epsilon$ are well-defined random field solutions to \eqref{E:SHE}.
Let $v_\epsilon(t,x) = u_\epsilon(t,x)-u(t,x)$
and $\tilde{\rho}(v_\epsilon) := \rho(v_\epsilon+u) - \rho(u)$.
It is clear that $\tilde{\rho}$ is a Lipschitz continuous function satisfying $\tilde{\rho}(0)=0$ {and $\Lip_{\tilde{\rho}}=\Lip_\rho$}.
Then $v_\epsilon$ is a solution to \eqref{E:SHE} with $\rho$ replaced by $\tilde{\rho}$
starting from $\mu_\epsilon:=\left((\mu \: \psi_\epsilon)*G(\epsilon,\cdot)\right)(x)-\mu$.
{Denote
\[
J_\epsilon(t,x)=(\mu_\epsilon*G(t,\cdot))(x)
\quad\text{and}\quad
g_\epsilon(t,x,x')=\left|\E\left[v_\epsilon(t,x)v_\epsilon(t,x')\right]
\right|.
\]
Then $g$ satisfies the following integral equation
\begin{align*}
g_\epsilon(t,x,x')\le & \quad |J_\epsilon(t,x)J_\epsilon(t,x')|\\
& + 
\Lip_\rho^2 \int_0^t\ud s\iint_{\R^{2d}}
G(t-s,x-y)G(t-s,x'-y')f(y-y') g(s,y,y')\ud y\ud y'.
\end{align*}
By Lemma \ref{L:TwoPoint}, we see that
\begin{align*}
g_\epsilon(t,x,x')\le & 
\quad |J_\epsilon(t,x)J_\epsilon(t,x')| \\
&+ 
C \int_0^t \ud s\iint_{\R^{2d}} G(t-s,x-y)G(t-s,x'-y')f(y-y') |J_\epsilon(s,y')J_\epsilon(s,y')| \ud y\ud y'.
\end{align*}
Notice that
\begin{align*}
\left|J_\epsilon(t,x)\right|
&\le
\left[
\left(|\mu\:\psi_\epsilon|*\left|G(t+\epsilon,\cdot)-G(t,\cdot)\right|\right)(x)
+
\left(|\mu\psi_\epsilon-\mu|*G(t,\cdot)\right)(x)
\right]\\
&\le \left[
\left(|\mu|*\left|G(t+\epsilon,\cdot)-G(t,\cdot)\right|\right)(x)
+
\left(|\mu\psi_\epsilon-\mu|*G(t,\cdot)\right)(x)
\right].
\end{align*}
Because for any $\epsilon\in (0,t)$, $\left|G(t+\epsilon,x)-G(t,x)\right|\le C G(2t,x)$ for all $x\in\R^d$ uniformly in $\epsilon$,
and because $|\mu\psi_\epsilon-\mu|\le |\mu|$,
we see that 
\[
\left|J_\epsilon(t,x)\right|
\le C \left(|\mu|*G(2t,\cdot)\right)(x) +
\left(|\mu|*G(t,\cdot)\right)(x).
\]
Then one can apply the dominated convergence theorem twice to conclude that
\[
\lim_{\epsilon\rightarrow 0} g_\epsilon(t,x,x') = 0,
\]
which completes the proof of part (1) of Theorem \ref{T:Approx}.}

{\bigskip\bf\noindent (2)~}
Since $u$ and $u_\epsilon$ start from the same initial data,
we see that 
\begin{align*}
\E \left[\left(u(t,x)-u_\epsilon (t,x) \right)^2\right] \leq& \quad 2 
\E \left(\int_0^t \int_{\RR^d} G(t-s,x-y)\left[\rho(u(s,y))-\rho(u_\epsilon (s,y))\right] M(\ud s,\ud y) \right)^2\\
&+2 \E \left( \int_0^t \int_{\RR^d} G(t-s,x-y) \rho(u_\epsilon (s,y)) \left( M(\ud s,\ud y)-M^{\epsilon}(\ud s,\ud y) \right) \right)^2\\
:=&I_1(t,\epsilon)+I_2(t,\epsilon)\,.
\end{align*}
For $I_1(t,\epsilon)$, using the Lipschitz condition on $\rho$ and since the initial condition is bounded, we obtain that
\begin{align*}
I_1(t,\epsilon)
&\leq C 
\int_0^t \int_{\RR^d} G(2(t-s),y) f(y)\sup_{z\in \RR^d} 
\E \left[\left(u(s,z)-u_\epsilon (s,z) \right)^2
\right]\ud y \ud s \\
&=C \int_0^t \ud s \: k(2(t-s))\sup_{z\in \RR^d} 
\E \left[\left(u(s,z)-u_\epsilon (s,z) \right)^2
\right]\,,
\end{align*}
where $k(\cdot)$ function is defined in \eqref{E:k}.
As for $I_2(t,\epsilon)$, we have that
\begin{align*}
&\E \left(\int_0^t \int_{\RR^d} G(t-s, x-y)\rho(u_{\epsilon}(s,y)) M(\ud s, \ud y) \int_0^t \int_{\RR^d} G(t-s, x-y) \rho(u_{\epsilon}(s,y))M^{\epsilon}(\ud s, \ud y) \right)\\
=&\E \bigg(\int_0^t \int_{\RR^d} G(t-s, x-y)\rho(u_{\epsilon}(s,y)) M(\ud s, \ud y)\\
&\quad \times \int_0^t \iint_{\RR^{2d}} G(t-s, x-y) \rho(u_{\epsilon}(s,y)) \phi_{\epsilon}(y-z)M(\ud s, \ud z) \ud y\bigg)\\
=&\E \bigg(\int_0^t \int_{\RR^d} G(t-s, x-y)\rho(u_{\epsilon}(s,y)) M(\ud s, \ud y)\\
&\quad \times \int_0^t \int_{\RR^d} \left(\int_{\RR^d} G(t-s, x-y) \rho(u_{\epsilon}(s,y)) \phi_{\epsilon}(y-z)\ud y\right)M(\ud s, \ud z) \bigg)\\
=&\E \int_0^t \ud s \iint_{\RR^{2d}}\ud y_1\ud y_2\: 
G(t-s, x-y_1)\rho(u_{\epsilon}(s,y_1)) G(t-s, x-y_2) \rho(u_{\epsilon}(s,y_2))\\
&\quad \times \int_{\R^d} \ud z \: \phi_{\epsilon}(y_2-z)  f(y_1-z)\\
= &\E \int_0^t\ud s \iint_{\RR^{2d}}
\ud y_1\ud y_2\: G(t-s, x-y_1)\rho(u_{\epsilon}(s,y_1)) G(t-s, x-y_2) \rho(u_{\epsilon}(s,y_2)) f^{\epsilon}(y_1-y_2)\,,
\end{align*}
where we have applied the stochastic Fubini theorem and  
$f^{\epsilon}(x):=\left(\phi_{\epsilon}* f\right)(x)$.
In the same way, we can get 
\begin{align*}
&\E \left[\left( \int_0^t \int_{\RR^d} G(t-s, x-y) \rho(u_{\epsilon}(s,y))M^{\epsilon}(\ud s, \ud y)\right)^2\right]\\
=& \E \int_0^t\ud s \iint_{\RR^{2d}}
\ud y_1\ud y_2\:
G(t-s, x-y_1)\rho(u_{\epsilon}(s,y_1)) G(t-s, x-y_2) \rho(u_{\epsilon}(s,y_2)) f^{\epsilon, \epsilon}(y_1-y_2),
\end{align*}
where $f^{\epsilon, \epsilon}(x):=\left(\phi_{\epsilon}*\phi_{\epsilon}*f\right)(x)$. 
Since $\phi$ is nonnegative definite, both $f^\epsilon$ and $f^{\epsilon,\epsilon}$ are
well-defined kernel functions.
From the above calculation, we see that the spatial correlation function for the noise $M^\epsilon$
is  $f^{\epsilon, \epsilon}(x)$. 
Notice that
\begin{align*}
k_\epsilon(t)&:=\int_{\R^d}f^{\epsilon,\epsilon}(z) G(t,z)\ud z
=(2\pi)^{-d}\int_{\R^d}\hat{f}(\ud \xi) \hat{\phi}_\epsilon(\xi)^2 \exp\left(-\frac{t|\xi|^2}{2}\right)\\
&\le
(2\pi)^{-d}\int_{\R^d}\hat{f}(\ud \xi) \exp\left(-\frac{t|\xi|^2}{2}\right)= k(t),
\end{align*}
for all $\epsilon>0$, where we have used the fact that 
\begin{align}
\label{E:hatphie}
\hat{\phi}_\epsilon(\xi)^2=
\hat{\phi}(\epsilon\xi)^2
=\left|\int_{\R^d}e^{-i\epsilon\InPrd{\xi,x}}\phi(x)\ud x\right|^2
\le \left(\int_{\R^d}\phi(x)\ud x\right)^2 =1.
\end{align}
Therefore, by Theorem \ref{T:Mom}, 
\[
\sup_{\epsilon>0}
\sup_{(s,x)\in[0,t]\times\R^d}\Norm{u_\epsilon (s,x)}_2\le 
\sup_{(s,x)\in[0,t]\times\R^d}\Norm{u(s,x)}_2<\infty.
\]
Thus, 
\begin{align*}
I_2(t,\epsilon)\leq& C \int_0^t \iint_{\RR^{2d}} G(t-s,x-y)G(t-s,x-z)
\\
&\times \left|f(y-z)-2f^{\epsilon}(y-z)+f^{\epsilon, \epsilon}(y-z)\right|\ud y \ud z \ud s\\
{=} &C \int_0^t \int_{\RR^d} G(2(t-s),y)|f(y)-2f^{\epsilon}(y)+f^{\epsilon, \epsilon}(y)|\ud y \ud s\\
\leq & C \int_{0}^t \int_{\RR^d}G(2(t-s),y)|f(y)-f^{\epsilon}(y)|\ud y \ud s \\
&+ C \int_{0}^t \int_{\RR^d}G(2(t-s),y)|f(y)-f^{\epsilon, \epsilon}(y)|\ud y \ud s\\
=&C \int_{\RR^d}g(2t,|y|)|f(y)-f^{\epsilon}(y)|\ud y 
+
C \int_{\RR^d}g(2t,|y|)|f(y)-f^{\epsilon, \epsilon}(y)|\ud y \,,
\end{align*}
where the function $g(t,|x|)$ is defined in Lemma \ref{L:gtx}.
Because $f$ is nonnegative and 
\[
\int_{\RR^d}g(4t,|y|)f(y)\ud y 
=\int_{0}^t \int_{\RR^d}G(4s,y)f(y)\ud y \ud s 
= \int_0^t k(4s) \ud s \le h_1(4t)<\infty, 
\]
part (2) of Lemma \ref{L: h appro} implies that
$\lim_{\epsilon \to 0} I_2(t,\epsilon)=0$.
Hence an application of Gronwall's lemma 
shows that 
\[
\lim_{\epsilon\rightarrow 0_+}\sup_{z\in \RR^d} 
\E \left[\left(u(t,z)-u_\epsilon (t,z) \right)^2
\right]=0, 
\]
which completes the proof of Theorem \ref{T:Approx}. 
\end{proof}

\section{A weak limit (Proof of Theorem \ref{T:WeakSol})}
\label{S:WeakSol}

\begin{proof}[Proof of Theorem \ref{T:WeakSol}]
Fix $\phi\in C_c(\R^d)$.
Let $I(t,x)$ be the stochastic integral part of \eqref{E:mild}.
We only need to prove that
\[
\lim_{t\rightarrow 0_+} \int_{\R^d} \ud x \: I(t,x) \phi(x) = 0 \quad\text{in $L^2(\Omega)$}.
\]
Denote $L(t):=\int_\R I(t,x) \phi(x)\ud x$.
By the stochastic Fubini theorem (see \cite[Theorem 2.6, p. 296]{Walsh}),
\[
L(t) =  \int_0^t \int_{\R^d} \left(\int_{\R^d} \ud x\;  G(t-s,x-y) \phi(x)\right)
\rho(u(s,y)) M(\ud s,\ud y).
\]
Hence, by \Itos isometry and the linear growth condition on $\rho$,
\begin{align*}
\E\left[L(t)^2\right]\le 
\Lip_\rho^2 \int_0^t \ud s & \iint_{\R^{2d}}  \ud y_1 \ud y_2 \: f(y_1-y_2) \iint_{\R^{2d}} \ud x_1\ud x_2 \: \\
&\times \sqrt{\Vip^2+\Norm{u(s,y_1)}_2^2} G(t-s,x_1-y_1)
|\phi(x_1)|\\
&\times \sqrt{\Vip^2+\Norm{u(s,y_2)}_2^2} G(t-s,x_2-y_2)|\phi(x_2)|\,,
\end{align*}
{where $\Vip=|\rho(0)|/\Lip_\rho$.}
Then by the moment bounds \eqref{E:Mom},
\begin{align*}
\E\left[L(t)^2\right]\le
C\int_0^t \ud s & \iint_{\R^{2d}}  \ud y_1 \ud y_2 \: f(y_1-y_2) \iint_{\R^{2d}} \ud x_1\ud x_2 \: \\
&\times \sqrt{1+J_0^2(s,y_1)}\: G(t-s,x_1-y_1)
|\phi(x_1)|\\
&\times \sqrt{1+J_0^2(s,y_2)}\: G(t-s,x_2-y_2)
|\phi(x_2)|\,.
\end{align*}
Assume that $t\le 1/2$.
By considering $\mu_*(\ud x)=\mu(\ud x) + \ud x$ and setting $J_*(t,x)=\left(\mu_**G(t,\cdot)\right)(x)$, we see that
\[
1+J_0^2(t,x) \le J_*^2(t,x).
\]
Because for some constant $C>0$, $|\phi(x)|\le C G(1,x)$ for all $x\in\R^d$,
we can apply the semigroup property to get
\begin{align*}
\E\left[L(t)^2\right]\le
C\int_0^t \ud s \iint_{\R^{2d}}  \ud y_1 \ud y_2 \: f(y_1-y_2) \:
       & J_*(s,y_1) G(t+1-s,y_1)\\
\times & J_*(s,y_2) G(t+1-s,y_2).
\end{align*}
Then by a similar argument as those in the proof of Lemma \ref{L:IntIneq},
we see that
\begin{align*}
\E\left[L(t)^2\right]\le & C J_*^2(t+1,x)\int_{\R^d}\hat{f}(\ud \xi)
\int_0^t \ud s\: \exp\left(-\frac{s(t+1-s)}{t+1}|\xi|^2\right)\\
\le &
C J_*^2(t+1,x)\int_{\R^d}\hat{f}(\ud \xi)
\int_0^t \ud s\: \exp\left(-\frac{s}{2}|\xi|^2\right),
\end{align*}
where the last inequality is due to $t\le 1/2$.
Since the above double integral is finite for $t=1/2$, by the dominated convergence theorem, we see that
this double integral goes to zero as $t\rightarrow 0$. This completes the proof of Theorem \ref{T:WeakSol}.
\end{proof}

\section{Weak comparison principle (Proof of Theorem \ref{T:WComp})}
\label{S:WComp}

\begin{proof}[Proof of Theorem \ref{T:WComp}]
We begin by noting that \eqref{E: W comp path} is 
an immediate consequence of \eqref{E: W comp point}. 
So we only need to prove \eqref{E: W comp point}. 
The proof consists of four steps.
Both the setup and Steps 1 \& 4 of the proof 
follow the same lines as those in the proof of Theorem 1.1 in \cite{CK14Comp}
with some minor changes. The main difference lies in Step 2 and Step 3.

Now we set up some notation in the proof. We view the $G(t,x)$ as an operator, denoted by ${\bf G}(t)$, as follows:
\begin{equation}
{\bf G}(t) f(x) := (G(t,\cdot)* f)(x)\,.
\end{equation}
Let $\bf I $ be the identity operator: ${\bf I} f(x):=(\delta * f)(x) = f(x)$. Set
\begin{equation}
\Delta ^{\epsilon} = \frac{{\bf G}(\epsilon) - {\bf I}}{\epsilon}\,.
\end{equation}
Let
\begin{equation}
G^{\epsilon}(t) = \exp (t\Delta^{\epsilon})= e^{-\frac{t}{\epsilon}} \sum_{n=0}^{\infty} \frac{(t/\epsilon)^n}{n!} {\bf G}(n\epsilon) := e^{-t/\epsilon} {\bf I} + {\bf R}^{\epsilon}(t)\,,
\end{equation}
where the operator ${\bf R}^{\epsilon}(t)$ has a density, denoted by $R^{\epsilon}(t,x)$, which is equal to
\begin{equation}\label{E:R epsilon}
R^{\epsilon}(t,x)=e^{-t/\epsilon} \sum_{n=1}^{\infty} \frac{(t/\epsilon)^n}{n!} G(n\epsilon,x)\,.
\end{equation}
For $\epsilon>0$ and $x \in \RR^d$, denote
\begin{equation}
M_x^{\epsilon}(t)=\int_0^t \int_{\RR^d} G(\epsilon,x-y) M(\ud s ,\ud y )\,,\quad \text{for} \ t \geq 0\,.
\end{equation}
Denote $\dot{M}_x^{\epsilon}(t) = \frac{\partial}{\partial t} M_x^{\epsilon}(t)$.
Then the quadratic variation of $\ud M_x^{\epsilon}(t)$ is
\begin{align*}
\ud \langle M_x^{\epsilon}(t)\rangle =& \iint_{\RR^{2d}} G(\epsilon,x-y_1) G(\epsilon,x-y_2) f(y_1-y_2)\ud y _1 \ud y _2 \ud t \\
=& \int_{\RR^d} e^{-\epsilon |\xi|^2} \hat{f}(\ud\xi) \ud t\,.
\end{align*}
Consider the following stochastic partial differential equation
\begin{equation}\label{E:approSHE}
 \left\{
\begin{array}{ll}
\displaystyle      \frac{\partial }{\partial t} u_{\epsilon}(t,x) = \Delta^{\epsilon} u_{\epsilon}(t,x) + \rho(u_{\epsilon}(t,x)) \dot{M}_{x}^{\epsilon}(t)\,, & t >0\,, x \in \RR^d\,,
\\[1em]
\displaystyle      u_{\epsilon}(0,x) = (\mu * G(\epsilon, \cdot))(x)\,, & x \in \RR^d \,. \\
\end{array}
\right.
\end{equation}
Since $\rho$ is Lipschitz continuous and $\Delta^{\epsilon}$ is a bounded operator,
\eqref{E:approSHE} has a unique strong solution
\begin{equation}
u_{\epsilon}(t,x)=
\left(\mu * G(\epsilon,\cdot)\right)(x) + \int_0^t \ud s  \Delta^{\epsilon} u_{\epsilon}(s,x)
+ \int_0^t \rho(u_{\epsilon}(s,x)) \ud M_x^{\epsilon}(s)\,.
\end{equation}
We proceed the proof in three steps. We fix $t>0$ and assume that $\epsilon\in (0,1\wedge t)$.

{\bigskip\noindent\bf Step 1}: Let $u_{\epsilon,1}(t,x)$ and $u_{\epsilon,2}(t,x)$ be the solutions to \eqref{E:approSHE}
with initial data $\mu_1$ and $\mu_2$, respectively.
Following exactly the same lines as those in Step 2 of the proof in \cite{CK14Comp}, we
can prove that $v_{\epsilon}(t,x):=u_{\epsilon,2}(t,x)-u_{\epsilon,1}(t,x)$ satisfies
\begin{equation}\label{eq: W-Comparison epsilon}
\bbP\Big(v_{\epsilon}(t,x)\geq 0, \ \text{for every } t > 0\ \text{and}\ x \in \RR^d\: \Big)=1\,.
\end{equation}
We will not repeat the proof here.

{\bigskip\noindent\bf Step 2}.
In this step we consider the case that 
the initial condition is bounded nonnegative function, 
i.e., $\mu(\ud x)=g(x)\ud x$ where $g(x)\ge 0$ and $g\in L^\infty(\R^d)$.
We also assume that the covariance function $f$ satisfies condition \eqref{E:DalangAlpha} with $\alpha=1$, i.e., 
\[
 \int_{\RR^d}\hat{f}(\ud\xi)< \infty.
\]
Let $u_\epsilon(t,x)$ be the solution to \eqref{E:SHE} starting from $u_{\epsilon}(0,x):= \left(\mu*G(\epsilon,\cdot)\right)(x)$.
The aim of this step is to prove 
\begin{equation}\label{eq: appro sol L2 conv}
\lim_{\epsilon \to 0} \sup_{x \in \RR^d} \|u_{\epsilon}(t,x)-u(t,x)\|_2^2 = 0\,, \quad\text{for all $t>0$}\,.
\end{equation}
Notice that $u_{\epsilon}(t,x)$ can be written in the following mild form using the kernel of ${\bf G}^{\epsilon}(t)$:
\begin{align*}
u_{\epsilon}(t,x)
& = \left(u_{\epsilon}(0,\cdot)* G^{\epsilon}(t,\cdot)\right)(x) + \int_0^t e^{-(t-s)/\epsilon} \rho(u_{\epsilon}(s,x))\ud M_x^{\epsilon}(s)\\
 &\quad \quad+ \int_0^t \int_{\RR^d} R^{\epsilon}(t-s,x-y) \rho(u_{\epsilon}(s,y))\ud M_y^{\epsilon}(s)\\
 &= \left(u_{\epsilon}(0,\cdot)* G^{\epsilon}(t,\cdot)\right)(x) + \int_0^t e^{-(t-s)/\epsilon} \rho(u_{\epsilon}(s,x))\ud M_x^{\epsilon}(s)\\
 & \quad \quad+ \int_0^t \int_{\RR^d} \left( \int_{\RR^d} R^{\epsilon}(t-s,x-z) \rho(u_{\epsilon}(s,z)) G(\epsilon,y-z) \ud z \right)  M(\ud s ,\ud y ).
\end{align*}
The boundedness of the initial data implies that 
\begin{equation}\label{E:BddMom}
A_t:= \sup_{\epsilon \in (0,1]} \sup_{s \in [0,t]} \sup_{x \in \RR^d} \|u_{\epsilon}(s,x)\|^2_2 \vee \|u(s,x)\|^2_2 < \infty\,.
\end{equation}
By the assumption on $\rho$, we have the following estimate:
\begin{align*}
\|u_{\epsilon}(t,x)- u(t,x)\|_2^2 \leq C \sum_{n=1}^6 I_n(t,x;\epsilon)\,,
\end{align*}
where
\begin{gather*}
 I_1(t,x;\epsilon) := \left( \left(u_{\epsilon}(0,\cdot)* G^{\epsilon}(t,\cdot)\right)(x)-u(0,\cdot)*G(t,\cdot)(x) \right)^2\,,\\
 I_2(t,x;\epsilon):= \int_0^t \ud s  \int_{\RR^d} e^{-\epsilon|\xi|^2} e^{-\frac{2(t-s)}{\epsilon}} \hat{f}(\ud\xi)\,,\\
 I_3(t,x;\epsilon):= \left\|\int_0^t \int_{\RR^d} \int_{\RR^d} R^{\epsilon}(t-s, x-z) \left[\rho(u_\epsilon (s,z)) - \rho(u(s,z))\right] G(\epsilon,y-z)\ud z M(\ud s ,\ud y ) \right\|_2^2,\\
 I_4(t,x;\epsilon):= \left\|\int_0^t \int_{\RR^d} \int_{\RR^d} R^{\epsilon}(t-s, x-z) \left[\rho(u(s,z)) - \rho(u(s,y))\right] G(\epsilon,y-z)\ud z M(\ud s ,\ud y ) \right\|_2^2,\\
I_5(t,x;\epsilon):= \left\|\int_0^t \int_{\RR^d} \int_{\RR^d} \left( R^{\epsilon}(t-s, x-z)-G(t-s,x-z)\right) \rho(u(s,y)) G(\epsilon,y-z) \ud z M(\ud s ,\ud y ) \right\|_2^2,
\end{gather*}
and
\[
I_6(t,x;\epsilon):= \left\|\int_0^t \int_{\RR^d} \int_{\RR^d} \left( G(t-s,x-y)-G(t-s,x-z)\right) \rho(u(s,y)) G(\epsilon,y-z) \ud z M(\ud s ,\ud y ) \right\|_2^2.
\]
Since $\mu$ has a bounded density, we see that
\begin{align*}
I_1(t,x; \epsilon)\le &C \left| \left(u_\epsilon(0,\cdot)* G^{\epsilon}(t,\cdot)\right)(x) - \left(u(0,\cdot)* G(t,\cdot)\right)(x) \right|\\
\leq &C \left(u_{\epsilon}(0,\cdot)* |G^{\epsilon}(t,\cdot)-G(t,\cdot)|\right)(x)
+
C \left(u(0,\cdot)*|G(t+\epsilon, \cdot)-G(t,\cdot)|\right)(x)\\
\leq& C \left(e^{-t/\epsilon}+ \int_{\RR^d} |R^{\epsilon}(t,y)-G(t,y)| dy + \int_{\RR^d} |G(t+\epsilon,y)-G(t,y)| dy \right).
\end{align*}
Then by Lemma \ref{lem:8.2} and the fact that $\log(1+x)\le \sqrt{x}$, we see that
\begin{align}
\sup_{x\in\R^d}  \sup_{s \in (0,t]} I_1(s,x; \epsilon) \le C
\left(e^{-t/\epsilon} + \sqrt{\epsilon/t}\right).
\end{align}
As for $I_2$, 
we see that 
\begin{align*}
I_2(t,x; \epsilon) &= \int_{\RR^d} e^{-\epsilon|\xi|^2} \frac{\epsilon}{2} (1-e^{-2t/\epsilon})\hat{f}(\ud\xi)\leq 
\frac{\epsilon}{2} \int_{\R^d} \hat{f}(\ud \xi)\le C\: \epsilon\,,
\end{align*}
which implies that
\begin{align}
\sup_{x\in\R^d}\sup_{s\in (0,t]} I_2(s,x; \epsilon) \le  C\: \epsilon\,.
\end{align}

The term $I_3$ will contribute to the recursion. By \eqref{E:BddMom},
\begin{align*}
I_3(t,x; \epsilon)&\leq \E \Bigg[
\int_0^t\ud s \iint_{\RR^{2d}}\ud y_1 \ud y_2\: f(y_1-y_2) \\
&\qquad\times\int_{\RR^d} \ud z_1\: R^{\epsilon}(t-s, x-z_1) \left[\rho(u_{\epsilon}(s,z_1))-\rho(u(s,z_1))\right]G(\epsilon,y_1-z_1) \\
&\qquad\times\int_{\RR^d} \ud z_2\: R^{\epsilon}(t-s, x-z_2) \left[\rho(u_{\epsilon}(s,z_2))-\rho(u(s,z_2))\right]G(\epsilon,y_2-z_2)\Bigg]
\\
&\leq C \int_0^t \ud s
\: \sup_{z\in \RR^d} \|u_{\epsilon}(s,z) - u(s,z)\|_2^2 
\iint_{\RR^{2d}} \ud y_1 \ud y_2\: f(y_1-y_2) \\
&\qquad \times \left(R^{\epsilon}(t-s, \cdot)* G(\epsilon,\cdot)\right)(x-y_1) \:
\left(R^{\epsilon}(t-s, \cdot)* G(\epsilon,\cdot)\right)(x-y_2) \\
& \leq C \int_0^t \ud s  \:\sup_{z\in \RR^d} \|u_{\epsilon}(s,z)-u(s,z)\|_2^2\,,
\end{align*}
where in the last line we used Lemma \ref{lem:mu inequality}.
As for $I_4$,
\begin{align*}
 I_4(t,x;\epsilon) &= \E \Bigg[ \int_0^t\ud s \iint_{\RR^{2d}}\ud y_1\ud y_2\: f(y_1-y_2)\\
 &\qquad\times
 \left(\int_{\RR^d} \ud z_1 \: R^{\epsilon}(t-s, x-z_1) \left[ \rho(u(s,z_1)) -\rho(u(s,y_1))\right] G(\epsilon,y_1-z_1) \right) \\
 &\qquad \times
 \left(\int_{\RR^d} \ud z_2 \: R^{\epsilon} (t-s, x-z_2) \left[\rho(u(s,z_2)) - \rho(u(s,y_2)) \right] G(\epsilon,y_2-z_2) \right)
 \Bigg]\\
 & \leq C \int_0^t\ud s \iint_{\RR^{2d}}\ud y_1\ud y_2\: f(y_1-y_2) \iint_{\RR^{2d}}\ud z_1\ud z_2 \\
 & \qquad \times R^{\epsilon}(t-s, x-z_1) \Norm{u(s,z_1)-u(s,y_1)}_2 G(\epsilon,y_1-z_1)\\
 & \qquad \times R^{\epsilon}(t-s, x-z_2) \Norm{u(s,z_2)-u(s,y_2)}_2 G(\epsilon,y_2-z_2).
\end{align*}
Then by the H\"older continuity of $u$ (see the proof of Theorem \ref{T:Holder}), we have that
\begin{align*}
I_4(t,x;\epsilon)
 & \leq C \int_0^t \ud s \iint_{\R^{2d}}\ud y_1\ud y_2 \: f(y_1-y_2) \iint_{\RR^{2d}} \ud z_1\ud z_2\\
 &\qquad\times R^{\epsilon}(t-s, x-z_1) |z_1-y_1|G(\epsilon,y_1-z_1)\\
 &\qquad\times  R^{\epsilon}(t-s, x-z_2)  |z_2-y_2|G(\epsilon,y_2-z_2)  \\
 & \leq C \epsilon \int_0^t\ud s \iint_{\RR^{2d}} \ud y_1\ud y_2\: f(y_1-y_2)
 \iint_{\R^{2d}}\ud z_1\ud z_2 \\
 &\qquad \times R^{\epsilon}(t-s, x-z_1)G(2\epsilon,y_1-z_1)   R^{\epsilon}(t-s, x-z_2) G(2\epsilon,y_2-z_2)\\
 &\leq C \epsilon \,,  
\end{align*}    
where {the last inequality is due to Lemma \ref{lem:mu inequality} and} the second inequality is due to{ the following inequality with $\alpha=1$:
\begin{equation}\label{E:aaGG}
|z_1-y_1|^\alpha |z_2-y_2|^\alpha G(\epsilon,y_1-z_1)G(\epsilon,y_2-z_2)
\leq C \epsilon^\alpha G(2\epsilon,y_1-z_1) G(2\epsilon,y_2-z_2) 
\,,
\end{equation}
for all $\alpha\in(0,1]$.}
Hence,  
\begin{align}
\sup_{x\in\R^d}\sup_{s\in[0,t]}I_4(s,x;\epsilon)\leq C \epsilon\,.
\end{align}
Now let's consider $I_5$,
\begin{align*}
 I_5(t,x;\epsilon) &= \E \Bigg[\int_0^t \ud s \iint_{\RR^{2d}}\ud y_1 \ud y_2\: f(y_1-y_2)\\
 &\qquad \times \left(\int_{\RR^d} \ud z_1 \left( R^{\epsilon}(t-s, x-z_1) - G(t-s, x-z_1) \right)\rho(u(s,y_1)) G(\epsilon,y_1-z_1) \right) \\
 &\qquad \times \left(\int_{\RR^d} \ud z_2 \left( R^{\epsilon} (t-s, x-z_2)- G(t-s, x-z_2) \right)\rho(u(s,z_2)) G(\epsilon,y_2-z_2) \right) \Bigg]\\
 &\leq C \int_0^t\ud s \iint_{\RR^{2d}} \ud y_1 \ud y_2 \:f(y_1-y_2) \iint_{\R^{2d}}\ud z_1\ud z_2\\
 & \qquad\times \left|R^{\epsilon}(t-s, x-z_1) - G(t-s, x-z_1) \right|G(\epsilon,y_1-z_1) \\
 &\qquad \times \left|R^{\epsilon}(t-s, x-z_2) - G(t-s, x-z_2) \right|  G(\epsilon,y_2-z_2) \\
 &\leq C \int_0^t\ud s \iint_{\R^{2d}}\ud z_1\ud z_2  \left|R^{\epsilon}(s, z_1) - G(s, z_1) \right| \: \left|R^{\epsilon}(s, z_2) - G(s, z_2) \right|\\
 & \qquad\times \iint_{\RR^{2d}} \ud y_1 \ud y_2 \:f(y_1-y_2)G(\epsilon,y_1-x+z_1) G(\epsilon,y_2-x+z_2) \\
 &= C \int_0^t\ud s \iint_{\R^{2d}}\ud z_1\ud z_2 \: \left|R^{\epsilon}(s, z_1) - G(s, z_1) \right| \: \left|R^{\epsilon}(s, z_2) - G(s, z_2) \right| f_{2\epsilon}(z_1-z_2),
\end{align*}
where
\[
f_{2\epsilon}(z) = (f* G(2\epsilon,\cdot))(z).
\]
Hence,
\begin{align*}
I_5(t,x;\epsilon)
&\le C \int_0^t\ud s \int_{\R^{d}}\ud z_1\ud z_2  \left|R^{\epsilon}(s, z_1) - G(s, z_1) \right| 
\int_{\R}\ud z_2 \:  \left(R^{\epsilon}(s, z_2) + G(s, z_2) \right) f_{2\epsilon}(z_1-z_2).
\end{align*}
Notice that by the assumption of $f$ in this step,  
\begin{align*}
\int_{\RR^{d}}\Big(R^{\epsilon}(s, z_2) + & G(s, z_2) \Big)  f_{2\epsilon}(z_1-z_2) \ud z _2 \\
 \leq&\int_{\RR^{d}}\Big(R^{\epsilon}(s, z_2) + G(s, z_2) \Big)  f_{2\epsilon}(z_2) \ud z _2\\
=& \int_{\RR^d} \left(e^{-s/\epsilon} \sum_{n=1}^{\infty} \frac{(s/\epsilon)^n}{n!} e^{-\frac{n\epsilon}{2}|\xi|^2} + e^{-\frac{s|\xi|^2}{2}} \right) e^{-\epsilon |\xi|^2} \hat{f}(\ud\xi) \leq C\,.
\end{align*}    
Thus, according to Lemma \ref{lem:8.2}, we have 
\begin{align*}
I_5(t,x;\epsilon)\leq& C \int_0^t \left( e^{-s/\epsilon} + \frac{\epsilon^{1/2}}{s^{1/2}} \right) 
 \leq C \epsilon^{1/2}\,.
\end{align*}    
Thus,  
\begin{align}
 \sup_{x\in\R^d}\sup_{s\in (0,t]} I_5(s,x; \epsilon) \le  C\: \epsilon^{1/2}.
\end{align}    
Now we study $I_6$.
By Lemma \ref{L:GG-xt},
\begin{align*}
I_6(t,x;\epsilon) &= \E \Bigg[\int_0^t\ud s \iint_{\RR^{2d}} \ud y_1 \ud y_2 \: f(y_1-y_2) \rho(u(s,y_1)) \rho(u(s,y_2))\\
&\qquad \times \left(\int_{\RR^d}  \ud z_1 \left(G(t-s, x-z_1) - G(t-s, x-y_1) \right) G(\epsilon,y_1-z_1) \right) \\
&\qquad \times\left(\int_{\RR^d}  \ud z_1 \left(G(t-s, x-z_2) - G(t-s, x-y_2) \right) G(\epsilon,y_2-z_2) \right)\Bigg]\\
& \leq C \int_0^t \ud s \iint_{\RR^{2d}} \ud y_1 \ud y_2 \: f(y_1-y_2) \iint_{\R^{2d}} \ud z_1\ud z_2  \\
&\qquad \times \left|G(t-s, x-z_1) - G(t-s, x-y_1) \right| G(\epsilon,y_1-z_1) \\
&\qquad \times  \left|G(t-s, x-z_2) - G(t-s, x-y_2) \right| G(\epsilon,y_2-z_2)\\
&\leq  C \int_0^t\ud s\:\frac{1}{(t-s)^{1/2}}  \iint_{\RR^{2d}}\ud y_1\ud y_2 \: f(y_1-y_2) \iint_{\R^{2d}} \ud z_1\ud z_2\\
&\qquad \times |z_1-y_1|^{1/2} \left[G(2(t-s),x-z_1)+G(2(t-s),x-y_1)\right] G(\epsilon,y_1-z_1) \\
&\qquad \times |z_2-y_2|^{1/2} \left[G(2(t-s),x-z_2)+G(2(t-s),x-y_2)\right] G(\epsilon,y_2-z_2).
\end{align*}
Then {by \eqref{E:aaGG} with $\alpha=1/2$ and by} the semigroup property, 
\begin{align*}
I_6(t,x;\epsilon) &\leq  C \epsilon^{1/2} \int_0^t\ud s\:\frac{1}{s^{1/2}}  \iint_{\RR^{2d}}\ud y_1\ud y_2 \: f(y_1-y_2) \iint_{\R^{2d}} \ud z_1\ud z_2\\
&\qquad \times \left[G(2s,x-z_1)+G(2s,x-y_1)\right] G(2\epsilon,y_1-z_1) \\
&\qquad \times \left[G(2s,x-z_2)+G(2s,x-y_2)\right] G(2\epsilon,y_2-z_2)\\
&=  C \epsilon^{1/2} \int_0^t\ud s\:\frac{1}{s^{1/2}}  \iint_{\RR^{2d}}\ud y_1\ud y_2 \: f(y_1-y_2) 
G(2(s+\epsilon),x-y_1)G(2(s+\epsilon),x-y_2)\\
&\le 
C \epsilon^{1/2} \int_0^\infty\ud s\:\frac{1}{s^{1/2}}  \int_{\R^{d}} e^{-2(s+\epsilon)(|\xi|^2+1)} \hat{f}(\ud \xi)\\
&\le C \epsilon^{1/2} \int_{\RR^d} \frac{\hat{f}(\ud\xi)}{(1+|\xi|^2)^{1/2}}\leq C \epsilon^{1/2}\,.
\end{align*}
Thus, 
\begin{align}
\sup_{x\in\R^d}\sup_{s\in[0,t]}I_6(s,x;\epsilon)\leq C \epsilon^{1/2}\,.
\end{align}
Therefore, by setting 
\[
M(t;\epsilon):=\sup_{y\in\R^d}\Norm{u_\epsilon(t,y)-u(t,y)}_2^2,
\]
we have shown that  
\[
M(t;\epsilon) \le C \left(\epsilon^{1/2}+e^{-t/\epsilon} + \sqrt{\epsilon/t}\right) +
C \int_0^t  M(s;\epsilon) \ud s.
\]
Then an application of Gronwall's lemma shows that
\[
M(t;\epsilon) \le 
 C \left(\epsilon^{1/2}+e^{-t/\epsilon} + \sqrt{\epsilon/t}\right) 
 + C e^{Ct}\int_0^t   \left(\epsilon^{1/2}+e^{-s/\epsilon} + \sqrt{\epsilon/s}\right)\ud s
 \rightarrow 0,\quad\text{as $\epsilon\rightarrow0$,}
\]
which proves \eqref{eq: appro sol L2 conv}.

{\bigskip\noindent\bf Step 3}
{
In this step we still work under the same assumption on the initial condition as in Step 2,
i.e., $\mu(\ud x)=g(x)\ud x$ with $g\ge 0$ and $g\in L^\infty(\R^d)$,
but we assume that the covariance function $f$ satisfies Dalang's condition \eqref{E:Dalang}. 
Choose a nonnegative and nonnegative definite function $\phi$ as in part (2) of 
Theorem \ref{T:Approx} (see also Remark \ref{R:Example}). 
Let $u(t,x)$ and $u_\epsilon(t,x)$ be the solutions to \eqref{E:SHE} and \eqref{E:SHE regularize},
respectively, with the same initial data $\mu$.
From the proof of part (2) of Theorem \ref{T:Approx},
we see that the spatial covariance function for $M^\epsilon$ is $\left(f*\phi_\epsilon*\phi_\epsilon\right)(x)$. 
We claim that $\left(f*\phi_\epsilon*\phi_\epsilon\right)(x)$ satisfies \eqref{E:DalangAlpha} with $\alpha=1$. Indeed, because $\phi(x)\le C G(1,x)$, we have that
$\phi_\epsilon(x) \le C G(\epsilon^2,x)$ and 
\begin{align*}
\int_{\R^d} \hat{f}(\ud\xi) \hat{\phi}_\epsilon(\xi)^2 
&=C 
\iint_{\R^{2d}} f(x-y) \phi_\epsilon(x)\phi_\epsilon(y)\ud x\ud y \\
&\le C \iint_{\R^{2d}} f(x-y) G(\epsilon^2,x)G(\epsilon^2,y)\ud x\ud y \\
&=C \int_{\R^{d}} f(y) G(2\epsilon^2,y)\ud y =C k(2\epsilon^2)<\infty,
\end{align*}
where $k(\cdot)$ is defined in \eqref{E:k}.
Hence, by Step 2, we see that
\[
\bbP\left(u_\epsilon(t,x)\ge 0\right)=1,\quad\text{for all $t>0$ and $x\in\R^d$.}
\]
Part (2) of Theorem \ref{T:Approx} implies that 
$u_\epsilon(t,x)$ converges to $u(t,x)$ a.s., for each $t>0$ and $x\in\R^d$.
Therefore, 
\[
\bbP\left(u(t,x)\ge 0\right)=1,\quad\text{for all $t>0$ and $x\in\R^d$.}
\]
Finally, suppose that $\mu_i(\ud x)=g_i(x)\ud x$ with $g_i\in L^\infty(\R^d)$, $i=1,2$. Let $u_{\epsilon, i}$ be the solutions of \eqref{E:SHE regularize} driven by $M^{\epsilon}$ and starting from initial conditions $\mu_i$.
If $g_1(x)\le g_2(x)$ for almost all $x\in\R^d$, then by Step 1, 
$v_\epsilon(t,x):= u_{\epsilon,2}(t,x)- u_{\epsilon,1}(t,x)\ge 0$ a.s. 
for all $t>0$ and $x\in\R^d$.
This step implies that
$v_\epsilon(t,x)$ converges to $v(t,x)=u_2(t,x) - u_1(t,x)$ in $L^2(\Omega)$ for all $t>0$ and $x\in\R^d$. 
Therefore,  $v(t,x)$ is nonnegative a.s., i.e., 
\[
\bbP\Big(u_1(t,x)\le u_2(t,x)\Big) =1\,, \:\text{for all $t>0$ and $x\in\R^d$}\,.
\]
}

{\bigskip\noindent\bf Step 4}.
Now we assume that the initial data $\mu_1$ and $\mu_2$ are measures that satisfy \eqref{E:J0finite}.
Recall the definition of $\psi_\epsilon$ in \eqref{E:psi}.
For $\epsilon>0$, let $u_{\epsilon,i}$, $i=1,2$, be the solutions to \eqref{E:SHE} starting from
$\left([\mu_i \psi_\epsilon]*G(\epsilon,\cdot)\right)(x)$.
Denote $v(t,x)=u_2(t,x)-u_1(t,x)$ and $v_\epsilon(t,x)=u_{\epsilon,2}(t,x)-u_{\epsilon,1}(t,x)$.
Because $\psi_\epsilon$ is a continuous function with compact support on $\R$,
the initial data for $u_{\epsilon,i}(t,x)$ are bounded functions.
By Step 3, we have that
\[
\bbP\Big(v_\epsilon(t,x)\ge 0\Big) =1,\quad\text{\:\text{for all $t>0$,  $x\in\R^d$} and  $\epsilon>0$.}
\]
Then part (1) of Theorem \ref{T:Approx} implies that
\[
\bbP\Big(v(t,x)\ge 0\Big) =1, \:\text{for all $t>0$ and $x\in\R^d$},
\]
which completes the whole proof of Theorem \ref{T:WComp}.
\end{proof}

\section{Strong comparison principle and strict positivity \\
(Proofs of Theorems \ref{T:SComp} \& \ref{T:Pos})}
\label{S:SComp}

We need some lemmas. Denote $Q(r)=[-r,r]^d$, i.e., a $d$-dimensional centered cube in $\RR^d$ of radius $r$.

\begin{lemma}\label{lem: comp initial}
Let $\ell>0$. For all $t >0$ and $M >0$, there exists some constants $1 < m_0=m_0(t,M)<\infty$ and $\gamma>0 $ such that for all $m \geq m_0$, $s \in \left[\frac{t}{2m}, \frac{t}{m}\right]$ and $x\in \RR^d$,
\begin{equation}
\left(G(s,\cdot)* \one_{Q(\ell)}\right)(x)\geq \gamma \one_{Q\left(\ell + \frac{M}{m}\right)}(x)\,.
\end{equation}
\end{lemma}

\begin{proof}
Since the $d$-dimensional heat kernel can be factored as a product of one-dimensional heat kernel, 
so the proof will be parallel with the proof of Lemma 4.1 in \cite{CK14Comp}. 
We will not repeat it here.
\end{proof}

\begin{lemma}\label{L:induction}
Let $\ell >0$, $t>0$, and $M > 0$. 
Assume that \eqref{E:DalangAlpha} holds for some $\alpha\in (0,1]$.
If $\rho(0)=0$ and $\mu(\ud x)= \one_{Q(\ell)}(x)\ud x$,
then there are some finite constants $\Theta:= \Theta(\beta, \LIP_{\rho}, t)>0$, $\beta >0$, and $m_0>0$ such that for all $m \geq m_0$,
\begin{align*}
\bbP\bigg( u(s,x) \geq \beta \one_{Q\left( \ell + \frac{M}{m}\right)}(x) \:\: &\text{for all $\frac{t}{2m}\leq s \leq \frac{t}{m}$ and $x\in\R^d$}\bigg)\\
&\geq 1- \exp \left( - \Theta m^{\alpha} \left( \log m\right)^{1+\alpha}\right)\,,
\end{align*}
where $\alpha\in (0,1]$ is the constant in \eqref{E:DalangAlpha}.
\end{lemma}
\begin{proof}
This proof follows similar arguments as those in the proof of Lemma 4.3 in \cite{CK14Comp}.
Here we only give a sketch of it.
Denote $S:=S_{t,m,\ell, M}:= \left\{(s,y): \frac{t}{2m} \leq s \leq \frac{t}{m},y \in Q\left(\ell + \frac{M}{m}\right)\right\}$.
By Lemma \ref{lem: comp initial}, for some constant $\beta >0$,
\begin{equation}
\left(\mu * G(s,\cdot)\right)(x)\geq 2\beta \one_{Q\left(\ell+ \frac{M}{m}\right)}(x) \quad \text{for all}\ s\in \left[\frac{t}{2m}, \frac{t}{m}\right]\ \text{and}\ x \in \RR^d\,.
\end{equation}
Then the stochastic integral part $I(t,x)$ of the mild solution in \eref{E:mild} satisfies
\begin{align*}
\bbP \Big(u(s,x)<& \:\beta \one_{Q( \ell+ \frac{M}{m})} \quad \text{for some } \ \frac{t}{2m}\leq s \leq \frac{t}{m}\ \text{and}\ x\in \RR^d\Big) \\
\leq&\bbP \Big( I(s,x)< -\beta \ \text{for some}\ (s,x)\in S\Big)\\
\leq&\bbP \left( \sup_{(s,x)\in S} |I(s,x)|> \beta\right) \leq \beta^{-p} \E \left( \sup_{(s,x)\in S} |I(s,x)|^p\right)\,.
\end{align*}
Denote $\tau = t/m$ and $S' := \left\{(s,y): 0 \leq s \leq t/m, |y|\leq \ell + M/m \right\}$. Using the fact that $I(0,x)\equiv 0$ for all $x\in \RR^d$, we see that for all $0 < \eta < 1-\frac{6d}{\alpha p}$,
\begin{align*}
\E \left( \sup_{(s,x)\in S} \left| \frac{I(s,x)}{\tau^{\frac{\alpha\eta}{2}}}\right|^p\right)
\leq&  \E \left( \sup_{(s,x), (s',x')\in S'} \left|\frac{I(s,x)-I(s',x')}{\left( |x-x'|^{\alpha}+|s-s'|^{\alpha/2}\right)^{\eta}}\right|^p\right)\,.
\end{align*}
We are interested in, and hence assume in the following, the case when $p=O([m\log m]^{\alpha})$ as $m\rightarrow\infty$;
see \eqref{E:Opt-p} below.
Since our initial condition is bounded, by \eqref{E:MomAlpha}, an application of the
Kolmogorov's continuity theorem shows that for large $p$,
\begin{align*}
\beta^{-p} \E \left( \sup_{(s,x)\in S} |I(s,x)|^p\right) \leq& C \tau^{\frac{\alpha}{2} p \eta}
e^{C p^{\frac{\alpha+1}{\alpha}}\tau}
\leq  C\exp \left( \frac{1}{2}\alpha p \eta \log \left( \tau\right)+ C p ^{\frac{\alpha+1}{\alpha}}\tau \right)\,.
\end{align*}
Since $p$ is large, we may choose $\eta=1/2$. Hence, the exponent in the right-hand side of the above inequalities becomes
\[
f(p):=
\frac{1}{4}\alpha p \log \left( \tau\right)+ C p ^{\frac{\alpha+1}{\alpha}} \tau.
\]
Some elementary calculation shows that $f(p)$ is minimized at
\begin{align}\label{E:Opt-p}
p=\left(\frac{\alpha ^2 \log (1/\tau)}{4(\alpha +1) C\tau }\right)^{\alpha}
=\left(\frac{\alpha ^2 m \log (m/t)}{4(\alpha +1) C t }\right)^{\alpha}.
\end{align}
Hence, for some positive constants $A$ and $\Theta$,
\[
\min_{p \ge 2} f(p)\le f(p') =
-\Theta m^\alpha \left[\log(m)\right]^{1+\alpha}\quad\text{with $p'=A \left[m\log(m)\right]^\alpha$.}
\]
This completes the proof of Lemma \ref{L:induction}.
\end{proof}

\bigskip
\begin{proof}[Proof of Theorem \ref{T:SComp}]
This proof follows the same arguments as those in the proof of Theorem 1.3 in \cite{CK14Comp}.
Here we only give a sketch of the proof. Interested readers are referred to \cite{CK14Comp} for details.

{Let $u(t,x):= u_2(t,x)-u_1(t,x)$ and denote $\tilde{\rho}(u)= \rho(u + u_1)-\rho(u_1)$. Then it is not hard to see that $u(t,x)$ is a solution to \eref{E:SHE} with the nonlinear function $\tilde{\rho}$ and the initial data $\mu:=\mu_2-\mu_1$. Note that $\tilde{\rho}$ is a Lipschitz continuous function with the same Lipschitz constant as for $\rho$ and $\tilde{\rho}(0)=0$. For simplicity, we will use $\rho$ instead of $\tilde{\rho}$. By the weak comparison principle, we only need to consider the case when $\mu$ has compact support and show that $u(t,x)>0$ for all $t>0$ and $x\in \RR^d$, a.s.
}

{\bigskip\bf\noindent Case I.~}
We fist assume that $\mu(\ud x)=\one_{Q(\ell)}\ud x$ for some $\ell>0$.
Denote
\begin{equation}\label{E:CM}
c(m):=\exp\left(- \Theta \: m^{\alpha} [\log(m)]^{1+\alpha}\right),
\end{equation}
where $\Theta$ is a constant defined in Lemma \ref{L:induction}.
We comment that due to a version mismatch in \cite{CK14Comp},
$B_0$ should be defined separately,  i.e.,
\begin{align*}
A_k &:= \left\{u(s,x)\ge \beta^{k+1} \one_{S_k^m}(x)\:\text{for all $s\in\left[\frac{(2k+1)t}{2m},\frac{(k+1)t}{m} \right]$ and $x\in\R^d$}\right\},\quad k\ge 0,\\
B_k &:= \left\{u(s,x)\ge \beta^{k+1} \one_{S_k^m}(x)\:\text{for all $s\in\left[\frac{kt}{m},\frac{(2k+1)t}{2m} \right]$ and $x\in\R^d$}\right\}, \quad k\ge 1, \\
B_0& :=\left\{u\left(\frac{t}{2m},x\right)\ge \beta \one_{S_0^m}(x)\:\text{for all $x\in\R^d$}\right\},
\end{align*}
where
{
\[
S^m_k :=\: \left (-\ell-\frac{Mk}{m},\ell+\frac{Mk}{m}\right)\:.
\]
}
See Figure \ref{F:SComp} for an illustration of the schema.

\begin{figure}
 \centering
 \includegraphics{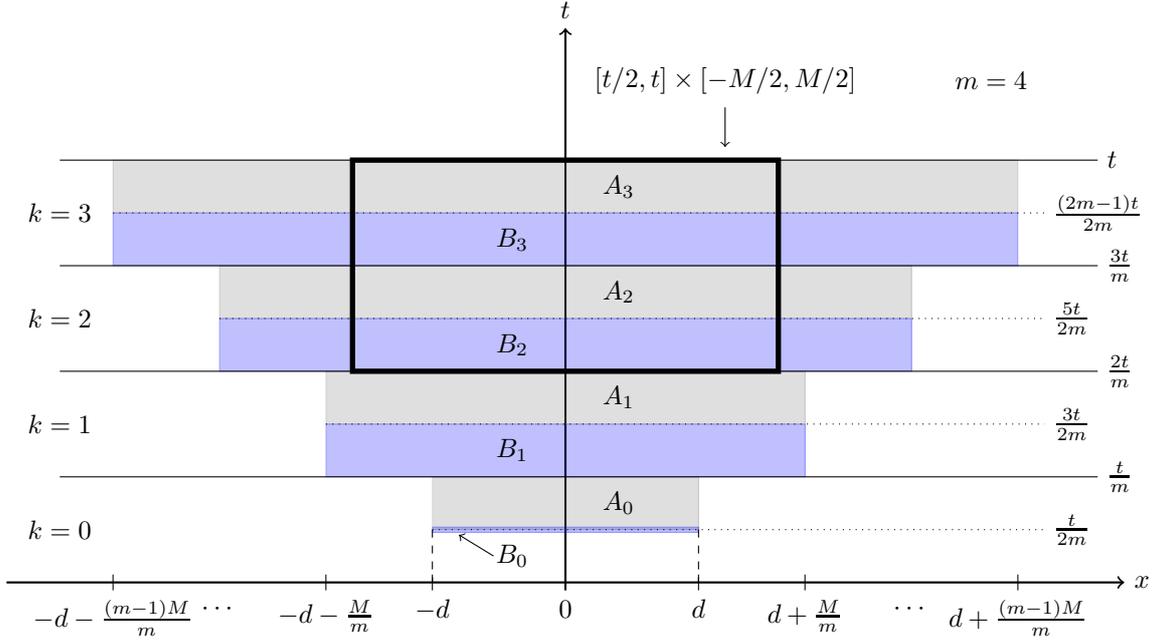}
 \caption{Induction schema for the strong comparison principle in the one-spatial dimension case.}
 \label{F:SComp}
\end{figure}

By an argument using the strong Markov property, one can show that
\[
\bbP\left(A_k\mid \calF_{kt/m}\right)\ge 1- c(m), \quad\text{a.s. on $A_{k-1}$ for $0\le k\le m-1$,}
\]
which implies 
\[
\bbP\left(A_k\mid A_{k-1}\cap \cdots \cap A_0\right) \ge 1-c(m),\quad\text{for all $1\le k\le m-1$.}
\]
Notice that the fact that $A_0\subseteq B_0$ implies that $\bbP(B_0) \ge \bbP(A_0) \ge 1-c(m)$.
By similar arguments as those for $A_k$, one can show that
\[
\bbP\left(B_k\mid B_{k-1}\cap \cdots \cap B_0\right) \ge 1-c(m),\quad\text{for all $1\le k\le m-1$.}
\]
Then,
\begin{align}\notag
\bbP\left(\cap_{0\le k\le m-1} \left[A_k \cap B_k\right] \right)
&\ge 1-\left(1-\bbP\left(\cap_{0\le k\le m-1}A_k\right)\right)-\left(1-\bbP\left(\cap_{0\le k\le m-1}B_k\right)\right)\\ \notag
&\ge (1-c(m))^{m-1}\bbP(A_0)+ (1-c(m))^{m-1}\bbP(B_0) -1\\
&\ge 2(1-c(m))^{m}-1.
\label{E_:abcup}
\end{align}
Therefore, for all $t>0$ and $M>0$,
\begin{align*}
 \bbP\Big(u(s,x)> 0 \;\; \text{for all $t/2\le s\le t$ and $x\in Q(M/2)$}\Big)
&\ge
\lim_{m\rightarrow\infty} \bbP\Big(\cap_{0\le k\le m-1} \left[A_k \cap B_k\right] \Big)
\\
&\ge \lim_{m\rightarrow\infty} 2(1-c(m))^{m}-1 =1.
\end{align*}
Since $t$ and $M$ are arbitrary, this completes the proof for the case when $\mu(\ud x)=\one_{Q(\ell)}\ud x$.

{\bigskip\bf\noindent Case II.~}
Now for general initial data $\mu$, we only need to prove that for each $\epsilon>0$,
\begin{align}\label{E:uEpsilon}
\bbP\left(u(t,x)>0\; \text{for $t\ge \epsilon$ and $x\in\R^d$}\right) =1.
\end{align}
Fix $\epsilon>0$. Denote $V(t,x):=u(t+\epsilon,x)$. By the Markov property, $V(t,x)$ solves \eref{E:SHE} with
the time-shifted noise $\dot{M}_\epsilon (t,x):=\dot{M}(t+\epsilon,x)$ starting from $V(0,x)=u(\epsilon,x)$, i.e.,
\begin{align}
 \label{E:VInt}
V(t,x) &= \left(u(\epsilon,\circ)* G(t,\cdot)\right)(x)
+ \iint_{[0,t]\times\R^d} \rho(V(s,y)) G(t-s,x-y)M_\epsilon(\ud s,\ud y).
\end{align}

We first prove by contradiction that
\begin{align}\label{E:uepsilonx}
\bbP\left(u(\epsilon,x)= 0, \;\text{for all $x\in\R^d$}\right) =0.
\end{align}
Notice that by Theorem \ref{T:Holder}, the function $x\mapsto u(t,x)$ is H\"older continuous over $\R^d$
a.s. The weak comparison principle (Theorem \ref{T:WComp}) shows that $u(t,x)\ge 0$ a.s.
Hence, if \eqref{E:uepsilonx} is not true, then by the Markov property and the strong comparison principle in Case I, at all times $\eta\in [0, \epsilon]$, with some strict positive probability,
$u(\eta,x)=0$ for all $x\in\R^d$, which contradicts Theorem \ref{T:WeakSol} as $\eta$ goes to zero.
Therefore, there exists a sample space $\Omega'$ with $\bbP(\Omega')=1$ such that for each $\omega\in\Omega'$,
there exists $x\in\R^d$ such that $u(\epsilon,x,\omega)>0$.

Since $u(\epsilon,x,\omega)$ is continuous at $x$, one can find two nonnegative constants $c=c(\omega)$ and $\beta=\beta(\omega)$ such that
$u(\epsilon,y,\omega)\ge \beta \one_{x+Q(c)}(y)$ for all $y\in\R^d$.
Then Case I implies that
\[
\bbP\left(V_{\omega}(t,x)>0\;\text{for all $t\ge 0$ and $x\in\R^d$}\right) =1,
\]
where $V_\omega$ is the solution to \eqref{E:VInt} starting from $u(\epsilon,x,\omega)$. Therefore, \eqref{E:uEpsilon} is true.
This completes the proof of Theorem \ref{T:SComp}.
\end{proof}

\bigskip
\begin{proof}[Proof of Theorem \ref{T:Pos}]
Following the proof of Theorem \ref{T:SComp}, since $K$ is compact, we can choose $\eta, T, N >0$ such that $K \subset [\eta, T]\times Q(N)$.  Let $\beta$, $A_k$ and $B_k$ be as in the proof of Theorem \ref{T:SComp},  we have
\begin{align*}
\bbP \left( \inf_{(t,x)\in K} u(t,x)< \beta^m\right) \leq& 1- \bbP \Big( \cap_{0 \leq k \leq m-1} (A_k \cap B_k)\Big)\\
\leq& 2 \left[1- (1-c(m))^m \right]\,,
\end{align*}
where $c(m)$ is a positive quantity defined in \eqref{E:CM}.
Then we use the fact that $(1-x)^m\geq 1-m x$ for all $x>0$ and $m>1$ to conclude that
for some $\Theta'$ slightly bigger than the $\Theta$ in \eqref{E:CM},
\begin{align*}
\bbP \left( \inf_{(t,x)\in K} u(t,x)< \beta^m\right) \leq& 2 m c(m)\leq \exp \left(- \Theta' m^{\alpha} \left(\log m \right)^{1+\alpha} \right)\,.
\end{align*}
Finally, by taking $m=|\log\epsilon|$, we complete the proof of Theorem \ref{T:Pos}.
\end{proof}

\appendix
\section{Appendix: Some technical lemmas}

Some technical lemmas are listed in this part.

\begin{lemma}\label{L:shatf}
If $g(t)$ is a monotone function over $[0,T]$, then for all $\beta>0$ and $t\in (0,T]$,
\begin{align}
\label{E:shatf1}
\int_0^t g(t-s)
\exp&\left(
-\frac{2 \beta s(t-s)}{t}\right)\ud s
=\int_0^t g(s)
\exp\left(-\frac{2\beta s(t-s)}{t}\right)\ud s\\[0.3em]
&\le
\begin{cases}
\displaystyle 2 \int_0^t g(s)e^{-\beta(t-s)}\ud s & \text{if $g$ is nondecreasing,}\\[0.8em]
\displaystyle  2 \int_0^t g(s)e^{-\beta s}\ud s & \text{if $g$ is nonincreasing.}
\end{cases}
\label{E:shatf2}
\end{align}
\end{lemma}
\begin{proof}
Equality \eqref{E:shatf1} is clear by change of variables.
We first assume that $g(t)$ is nondecreasing in $[0,T]$.
Denote the integral by $I$. Then
\begin{align*}
I=&\int_0^{t/2}
 g(s) \exp\left(- \frac{2\beta s(t-s)}{t}\right)
 \ud s
 +\int_{t/2}^t
 g(s) \exp\left(- \frac{2\beta s(t-s)}{t}\right)
 \ud s\\
 \le &
 \int_0^{t/2}
 g(s) \exp\left(- \beta s\right)
 \ud s
 +\int_{t/2}^t
 g(s) \exp\left(- \beta(t-s)\right)
 \ud s\\
 \le &
 \int_{t/2}^t
 g(t-s) \exp\left(- \beta (t-s)\right)
 \ud s
 +\int_{t/2}^t
 g(s) \exp\left(- \beta(t-s)\right)
 \ud s\\
 \le &
 2\int_{t/2}^t
 g(s) \exp\left(- \beta(t-s)\right)
 \ud s\\
 \le &
 2\int_{0}^t
 g(s) \exp\left(- \beta (t-s)\right)
 \ud s.
\end{align*}
If $g$ is nonincreasing in $[0,T]$,
we simply replace the above $g(s)$ by $g(t-s)$ thanks to \eqref{E:shatf1}.
This proves Lemma \ref{L:shatf}.
\end{proof}


\begin{lemma}\label{lem:mu inequality}
 Let $R^{\epsilon}$ be defined in \eqref{E:R epsilon}.
 {If $f$ satisfies \eqref{E:DalangAlpha} with $\alpha=1$,}
 then there exists a positive constant $C$ such that for all $0\le s, \epsilon \le t$ and $x\in\R^d$,
\begin{align*}
\iint_{\RR^{2d}}\ud y_1\ud y_2\: f(y_1-y_2) & \left(R^{\epsilon}(t-s,\cdot)*G(\epsilon,\cdot)\right)(x-y_1)\\
\times&\left(R^{\epsilon}(t-s,\cdot)*G(\epsilon,\cdot)\right)(x-y_2)  \leq C \,.
\end{align*}
\end{lemma}
\begin{proof}
Denote the integral by $I$. Using Fourier transform
we have
\begin{align*}
I&\leq \int_{\RR^d} e^{-\frac{2(t-s)}{\epsilon}} \sum_{n,m=1}^{\infty} \frac{\left(\frac{t-s}{\epsilon}\right)^n}{n!}\frac{\left(\frac{t-s}{\epsilon}\right)^m}{m!} e^{-\frac{(n+m)\epsilon}{2} |\xi|^2} \hat{f}(\ud\xi)\\
&\leq C e^{-\frac{2(t-s)}{\epsilon}} \sum_{n,m=1}^{\infty} \left(\frac{t-s}{\epsilon}\right)^{m+n} 
\frac{1}{n!m!} \,.
\end{align*}
Letting $n+m=k$ and using the fact that
\begin{equation*}
\sum_{n=1}^{k-1} \frac{1}{n! (k-n)!} = \frac{1}{k!} (2^k-2)\,,
\end{equation*}
we see  that the above double sum is equal to 
\begin{align*}
 \sum_{k=1}^{\infty} \sum_{n=1}^{k-1} \left(\frac{t-s}{\epsilon}\right)^k \frac{1}{n! (k-n)!} 
&\leq  \sum_{k=1}^{\infty} \left(\frac{t-s}{\epsilon}\right)^k \frac{2^k}{k!} 
\le e^{\frac{2(t-s)}{\epsilon}} - 1\,,
\end{align*}     
which proves Lemma \ref{lem:mu inequality}.
\end{proof}

\begin{lemma}\label{lem:8.2}
There exists a finite constant $C > 0$ such that
\begin{equation}
\int_{\RR^d} |R^{\epsilon}(t,x)-G(t,x)|\ud x \leq e^{-t/\epsilon}+ C \left(\frac{\epsilon}{t}\right)^{1/2}\,, 
\end{equation}
and
\begin{equation}
\int_{\RR^d} |G(t+\epsilon,x)-G(t,x)|\ud x \leq C \log\left(1+\frac{\epsilon}{t}\right)\,,
\end{equation}
for all $\epsilon>0$ and $t>0$.
\end{lemma}
\begin{proof}
Because $\left | \frac{\partial }{\partial t} G(t,x) \right| \leq C t^{-1} G(2t, x)$, we see that
for any $0 < t \leq t'$,
\begin{align*}
\int_{\RR^d} |G(t', x)-G(t,x)|\ud x &\leq \int_{\RR^d} \ud x
\int_t^{t'} \ud s |\frac{\partial }{\partial s} G(s,x)|\\
& \leq C \int_{\RR^d} \ud x \int_t^{t'} \ud s\: s^{-1} G(2s, x) \\
& \leq C \log \left(t'/t\right).
\end{align*}
The rest of the proof will follow exactly the same lines as those in the proof of Lemma 8.2 in \cite{CK14Comp}
and we will not repeat here.
\end{proof}

\begin{lemma}\label{L:gtx}
The function $g(t,x):=\int_0^t (2\pi s)^{-d/2} \exp\left(-\frac{x^2}{2s}\right)\ud s$, for $t,x\ge 0$, satisfies the following properties,
\begin{enumerate}[(1)]
\item $x\mapsto g(t,x)$ is strictly decreasing functions on $x\in(0,\infty)$.
\item If $d=1$, then $g(t,x)$ doesn't blow up at $x=0$ and $g(t,x)\le g(t,0)= \sqrt{2t/\pi}$.
If $d\ge 2$, then $g(t,x)$ blows up at $x=0$.
\item If $d=1,2$, then for all $\theta >0$ and $t>0$,
\begin{align}\label{E:intg}
\int_{\R^d } g(t,|x| )^\theta \ud x <\infty.
\end{align}
\item If $d\ge 3$, then for all $0< \theta<\frac{d}{d-2}$ and $t>0$, \eqref{E:intg} holds.
\end{enumerate}
\end{lemma}
\begin{proof}
(1) It is clear $x\mapsto g(t,x)$ is a nonincreasing function on $(0,\infty)$ because
\[
\frac{\partial }{\partial x} g(t,x)=
- \int_0^t (2\pi s)^{-d/2} \frac{x}{s}\exp\left(-\frac{x^2}{2s}\right)\ud s< 0,\quad\text{for $x>0$.}
\]
(2) If $d=1$, then by (1), we see that $g(t,x)\le g(t,0)=\sqrt{2t/\pi}$.
By change of variables $z=x^2/(2s)$,
\begin{align}\label{E:gtx}
g(t,x)= \frac{1}{2\pi^{d/2}} x^{2-d} \int_{\frac{x^2}{2t}}^\infty e^{-z} z^{\frac{d}{2}-2} \ud z.
\end{align}
If $d=2$, then the integral in \eqref{E:gtx} blows up as $x\rightarrow 0_+$.
When $d\ge 3$,
\begin{align}
\label{E:gtx_bds}
g(t,x)\le
\frac{1}{2\pi^{d/2}} x^{2-d} \int_{0}^\infty e^{-z} z^{\frac{d}{2}-2} \ud z =
\frac{\Gamma(d/2-1)}{2\pi^{d/2}} x^{2-d},
\end{align}
which blows up as $x\rightarrow 0_+$. \\
(3) 
 { If $d=1$, for all $t>0$ and $x\ge 0$,
 \begin{align*}
 g(t,x)\le& \frac{1}{\sqrt{2\pi}}  e^{-\frac{x^2}{2t}} \int_{0}^{t} \frac{1}{\sqrt{s}}\ud s
 = \frac{\sqrt{2 t}}{\sqrt{\pi}}  e^{-\frac{x^2}{2t}}\,,
 \end{align*}
which shows \eqref{E:intg} for $d=1$. 
}
If $d=2$, then
\[
g(t,x) = \frac{1}{2\pi}\int_{x^2/(2t)}^\infty e^{-z}z^{-1}\ud z.
\]
Then by l'Hopital's rule,
\[
\lim_{x\rightarrow 0_+} \frac{g(t,x)}{\log(1/x)} =
\frac{1}{2\pi}\lim_{x\rightarrow 0_+}\frac{-e^{-\frac{x^2}{2t}}\frac{2t}{x^2}\: \frac{x}{t}}{-1/x}=\frac{1}{\pi}.
\]
While for $x \geq 1$, 
\begin{align*}
g(t,x)=& \frac{1}{2\pi}\int_{\frac{x^2}{2t}}^{\infty} e^{-z} z^{-\frac{3}{2}}\ud z\leq \frac{1}{2\pi} \left( \frac{x^2}{2t}\right)^{-\frac{3}{2}} \int_{\frac{x^2}{2t}}^{\infty} e^{-z} \ud z \leq \frac{(2t)^{3/2}}{2\pi}e^{-\frac{x^2}{2t}}\,,
\end{align*}
from which we can conclude (3).\\
(4) For $d\geq 3$, note that there is a constant $C_d> 0$ which only depends on $d$ such that $z^{\frac{d}{2}-2} e^{-z} \leq C_d e^{-\frac{z}{2}}$ for all $z\geq 0$. Then for $x\geq 1$, 
\begin{align*}
g(t,x)= \frac{1}{2\pi^{d/2}} x^{2-d} \int_{\frac{x^2}{2t}}^{\infty} e^{-z} z^{\frac{d}{2}-2}\ud z \leq \frac{C_d}{2\pi^{d/2}} \int_{\frac{x^2}{2t}}^{\infty} e^{-\frac{z}{2}}\ud z \leq \frac{C_d}{\pi^{d/2}} e^{-\frac{x^2}{4t}}\,,
\end{align*}
this shows that for any $\theta >0$, 
\begin{equation}
\int_{|x|\geq 1} g(t,|x|)^{\theta} \ud x < \infty\,.
\end{equation}
The restriction that $\theta < \frac{d}{d-2}$ comes from the integrability on $|x|\leq 1$, which is clear from the upper bound of $g(t,x)$ in \eqref{E:gtx_bds}.
This completes the proof of Lemma \ref{L:gtx}.
\end{proof}

\begin{lemma}\label{L: h appro}
Recall the function $g(t,x)$ is defined in Lemma \ref{L:gtx}.
{Let $\psi\in C_c(\R^d)$ be an arbitrary mollifier}
such that $\int_{\R^d}\psi(x)\ud x=1$.
Denote 
$\psi_\epsilon(x) = \epsilon^{-d} \psi(x/\epsilon)$. For each fixed $t>0$,
suppose that $h:\R^d\mapsto\R_+$ is a nonnegative and measurable
function such that
\[
\int_{\R^d} h(x)g(2t,|x|)\ud x <\infty.
\]
Then the following statements hold:
\begin{enumerate}[(1)]
 \item For any $\eta>0$, there exists $\phi\in C_c(\R^d)$
 such that 
\[
\sup_{\epsilon\in (0,\sqrt{t})}\int_{\R^d} g_\epsilon(t,|x|)\left|h(x) - \phi(x)\right|\ud x <\eta,
\]
where $g_\epsilon(t,|x|) = \int_{\R^d} g(t,|y|)\psi_\epsilon(x-y)\ud y$.
 \item By denoting $h_\epsilon(x)=(h*\psi_\epsilon)(x)$, we have that
 \[
\lim_{\epsilon\rightarrow 0}
\int_{\R^d}{g(t,|x|)}\left|h(x)-h_\epsilon(x)\right|\ud x =0.
\]
\end{enumerate}
\end{lemma}
\begin{proof}
Without loss of generality, we may assume that $t=1$. 

{\noindent(1)~} 
Fix $\eta>0$.
It is clear that for some constant $C>0$, we have 
\[
\psi(x)\le C G(1,x),\quad\text{for all $x\in\R^d$.}
\]
Hence, $\psi_\epsilon(x)\le C G(\epsilon^2,x)$, which implies that
\begin{align}
g_\epsilon(1,|x|) 
& \le C\int_{\R^d}\ud y\:
G(\epsilon^2,x-y) \int_{0}^1\ud s\: G(s,y)\notag\\
&= C  \int_{0}^1\ud s \:G(s+\epsilon^2,x)\notag\\
&= C  \int_{\epsilon^2}^{1+\epsilon^2}\ud s \:G(s,x)\le C g(2, |x|)\,,\label{E:g epsilon bd}
\end{align}
where the last inequality is due to the definition of $g(t,x)$ and $\epsilon\in (0,1)$.
Since $h$ is nonnegative, it is known that one can 
find a monotone nondecreasing sequence $\{s_j\}$
of simple functions such that $s_j(x)\uparrow h(x)$
pointwise; see, e.g., Theorem 1.44 in \cite{Adam03SecondEd}.
Hence,
\begin{align*}
\sup_{\epsilon\in(0,1)}\int_{\R^d}g_\epsilon(1,|x|)\left|h(x) - s_j(x)\right|\ud x 
&<
C\int_{\R^d}g(2,|x|)\left|h(x) - s_j(x)\right|\ud x
\rightarrow\infty 
\end{align*}
as $j\rightarrow\infty$, where the last limit is 
due to the dominated convergence.
Hence, for some $s\in \{s_j\}$, 
\[
\sup_{\epsilon\in(0,1)}\int_{\R^d}g_\epsilon(1,|x|)\left|h(x) - {s(x)}\right|\ud x 
\le \eta/2.
\]
Now we choose and fix $q>1$ such that 
\begin{equation}
C(g, d, q):=\int_{\RR^d} g(t,|x|)^q \ud x < \infty\,.
\end{equation}
This is possible thanks to Lemma \ref{L:gtx}:
$q> 1$ can be any number for $d=1,2$ and $ q\in (1, \frac{d}{d-2})$ for $d\ge 3$. 
Since $s$ is a simple function with bounded support,
by Lusin's theorem (see e.g., Theorem 1.42 (f) in \cite{Adam03SecondEd})
there exists $\phi\in C_c(\R^d)$ such that
\[
|\phi(x)|\le \Norm{s}_{L^\infty(\R^d)},\quad\text{for all $x\in\R^d$}
\]
and  
\[
\text{Vol}\left(\left\{
x\in\R^d:\: \phi(x)\ne s(x)\right\}
\right)\le
\frac{\eta^p}{\left(4C \Norm{s}_{L^\infty(\R^d)} C(g,d,q)^{1/q}\right)^p}\,,
\]
where $1/p+ 1/q=1$ and $C$ is as in \eqref{E:g epsilon bd}.
Thus, using \eqref{E:g epsilon bd} and the H\"older inequality, 
\begin{align*}
\sup_{\epsilon\in (0,1)}&\int_{\R^d} g_\epsilon(1,|x|)
\left|s(x) - \phi(x)\right|\ud x\\
&\le 
C \int_{\R^d} g(2,|x|)
\left|s(x) - \phi(x)\right|\ud x\\
&\le 2 C \Norm{s}_{L^\infty(\R^d)}
\int_{\R^d}
\one_{\left\{
x\in\R^d:\: \phi(x)\ne s(x)\right\}}g(2,|x|)\ud x\\
&\leq 2 C \Norm{s}_{L^\infty(\R^d)} \left( \int_{\RR^d} \one_{\left\{ x \in \RR^d: \: \phi(x)\ne s(x) \right\}} \ud x\right)^{\frac{1}{p}} \left(\int_{\RR^d} g(2, |x|)^q\ud x \right)^{\frac{1}{q}}\\
&\leq \frac{\eta}{2}\,.
\end{align*}
This completes the proof of (1). 

{\bigskip\noindent(2)~} For any $\eta>0$, 
we can write 
\begin{align*}
\int_{\RR^d} |h_{\epsilon}(x)-h(x)| g(1,|x|)\ud x
=&\int_{\RR^d} \left| \int_{\RR^d} \psi_{\epsilon}(x-y)\left[h(y)-h(x)\right]\ud y\right| g(1,|x|)\ud x\\
=&\int_{\RR^d} \left| \int_{\RR^d} \psi_{\epsilon}(x-y)\left[h(y)-\phi(y)\right]\ud y\right| g(1,|x|)\ud x\\
&+\int_{\RR^d} \left| \int_{\RR^d} \psi_{\epsilon}(x-y)\left[\phi(y)-\phi(x)\right]\ud y\right| g(1,|x|)\ud x\\
&+ \int_{\RR^d} \left| \int_{\RR^d} \psi_{\epsilon}(x-y)\left[\phi(x)-h(x)\right]\ud y\right| g(1,|x|)\ud x\\
:=& I_1 + I_2 + I_3\,.
\end{align*}
For $I_1$, choose $\phi \in C_c(\RR^d)$ according to (1), such that $I_1< \frac{\eta}{3}$. 
From the proof of (1) it is obvious that with the same choice of $\phi$, $I_3< \frac{\eta}{3}$.
For $I_2$, since $\psi$ is compactly supported, we may choose $\epsilon_0>0$ 
such that whenever $0 < \epsilon < \epsilon_0$, 
we have $I_2< \frac{\eta}{3}$ because of the uniform continuity of $\phi$.  
This completes the proof of (2). 
\end{proof}

\section*{Acknowledgements}
\addcontentsline{toc}{section}{Acknowledgements}
Both authors appreciate some stimulating discussions with Davar Khoshnevisan. The first author also thanks Kunwoo Kim for many interesting discussions when both of them were in Utah in the year of 2014.

\begin{small}
\bigskip

\begin{minipage}{0.4\textwidth}
\noindent\textbf{Le Chen}\\
\noindent Department of Mathematics\\
\noindent University of Kansas\\
\noindent Lawrence, Kansas, 66045-7594\\
\noindent\emph{Email:} \texttt{chenle@ku.edu}\\
\noindent\emph{URL:} \url{www.math.ku.edu/u/chenle/}\\
\end{minipage}
\begin{minipage}{0.6\textwidth}
\noindent\textbf{Jingyu Huang}\\
\noindent Department of Mathematics\\
\noindent University of Utah\\
\noindent Salt Lake City, UT 84112-0090\\
\noindent\emph{Email:} \texttt{jhuang@math.utah.edu}\\
\noindent\emph{URL:} \url{http://www.math.utah.edu/~jhuang/}\\
\end{minipage}

\end{small}

\end{document}